\documentclass{amsart}
\usepackage[dvips,final]{graphics}
\usepackage{array}
\usepackage{arydshln}
\usepackage[makeroom]{cancel}
 \usepackage[all]{xy}
 \usepackage{url}
\usepackage{multirow, blkarray}
\usepackage{booktabs}
\usepackage{textcomp}
 \usepackage[final]{epsfig}
 \usepackage{color}
\usepackage[T1]{fontenc}      
\usepackage[english,french]{babel}
\usepackage[utf8]{inputenc}
\usepackage{blindtext}

\usepackage{amsfonts,amscd,array, mathdots, epigraph}
\usepackage{amsmath}
\usepackage{amssymb}
\usepackage{amsthm}
\usepackage{mathrsfs}
\usepackage{stmaryrd}
\usepackage{slashbox}
\usepackage{diagbox}
\usepackage{enumitem}

\usepackage{ulem}
\usepackage{tikz}
\usepackage{xcolor}
\usepackage{multicol}
\definecolor{ufogreen}{rgb}{0.24, 0.82, 0.44}

\begin{document}


\newtheorem{theorem}{Théorème}[section]
\newtheorem{theore}{Théorème}
\newtheorem{definition}[theorem]{Définition}
\newtheorem{proposition}[theorem]{Proposition}
\newtheorem{corollary}[theorem]{Corollaire}
\newtheorem*{con}{Conjecture}
\newtheorem*{remark}{Remarque}
\newtheorem*{remarks}{Remarques}
\newtheorem*{pro}{Problème}
\newtheorem*{examples}{Exemples}
\newtheorem*{example}{Exemple}
\newtheorem{lemma}[theorem]{Lemme}


\title{Entiers monomialement irréductibles}

\author{Flavien Mabilat}

\date{}


\address{
Flavien Mabilat,
Laboratoire de Mathématiques de Reims,
UMR9008 CNRS et Université de Reims Champagne-Ardenne, 
U.F.R. Sciences Exactes et Naturelles 
Moulin de la Housse - BP 1039 
51687 Reims cedex 2,
France
}
\email{flavien.mabilat@univ-reims.fr}

\maketitle

\selectlanguage{french}
\begin{abstract}
Dans cette article on va s'intéresser à la combinatoire des sous-groupes de congruence du groupe modulaire. On va se consacrer ici à la notion de solutions monomiales minimales. Celles-ci sont les solutions d'une équation matricielle (apparaissant également lors de l'étude des frises de Coxeter), modulo un entier $N$, dont toutes les composantes sont identiques et minimales pour cette propriété. L'objectif ici est d'étudier les entiers $N$ pour lesquels toutes ces solutions possèdent une certaine propriété d'irréductibilité. 
\\
\end{abstract}

\selectlanguage{english}
\begin{abstract}
In this article we study the combinatorics of congruence subgroups of the modular group. We consider the notion of minimal monomial solutions. These are the solutions of a matrix equation (also appearing in the study of Coxeter friezes), modulo an integer $N$, whose components are identical and minimal for this property. Our objective is to study the integers $N$ for which all these solutions have a certain irreducibility property.

\end{abstract}

\selectlanguage{french}

\thispagestyle{empty}

\noindent {\bf Mots clés:} groupe modulaire; solution monomiale; irréductibilité 
\\
\begin{flushright}
\og \textit{La connaissance est une navigation dans un océan d'incertitudes à travers des archipels de certitudes.} \fg
\\Edgar Morin, \textit{Les sept savoirs nécessaires à l'éducation du futur.}
\end{flushright}

\section{Introduction}

Dans de nombreux domaines mathématiques, on est amené à considérer des matrices de la forme suivante : \[M_{n}(a_{1},\ldots,a_{n})=\begin{pmatrix}
   a_{n} & -1 \\[4pt]
    1    & 0 
   \end{pmatrix}
\begin{pmatrix}
   a_{n-1} & -1 \\[4pt]
    1    & 0 
   \end{pmatrix}
   \cdots
   \begin{pmatrix}
   a_{1} & -1 \\[4pt]
    1    & 0 
    \end{pmatrix}.\]
En effet, celles-ci apparaissent dans l'étude d'un nombre importants d'objets comme les fractions continues négatives (appelées également fractions continues de Hirzebruch-Jung), les frises de Coxeter, l'équation de Sturm-Liouville discrète (voir \cite{O} section 1.3) et bien d'autres encore. Un de ces objets, particulièrement intéressant, est le groupe modulaire \[SL_{2}(\mathbb{Z})=
\left\{
\begin{pmatrix}
a & b \\
c & d
   \end{pmatrix}
 \;\vert\;a,b,c,d \in \mathbb{Z},\;
 ad-bc=1
\right\}.\] En effet, considérons la partie génératrice formée des deux éléments suivants
 \[T=\begin{pmatrix}
 1 & 1 \\[2pt]
    0    & 1 
   \end{pmatrix}, S=\begin{pmatrix}
   0 & -1 \\[2pt]
    1    & 0 
   \end{pmatrix}.
 \] 
\\On peut montrer (voir par exemple l'introduction de \cite{Ma3}) que pour toute matrice $A$ de $SL_{2}(\mathbb{Z})$ il existe un entier strictement positif $n$ et des entiers strictement positifs $a_{1},\ldots,a_{n}$ tels que \[A=T^{a_{n}}ST^{a_{n-1}}S\cdots T^{a_{1}}S=\begin{pmatrix}
   a_{n} & -1 \\[4pt]
    1    & 0 
   \end{pmatrix}
\begin{pmatrix}
   a_{n-1} & -1 \\[4pt]
    1    & 0 
   \end{pmatrix}
   \cdots
   \begin{pmatrix}
   a_{1} & -1 \\[4pt]
    1    & 0 
    \end{pmatrix}=M_{n}(a_{1},\ldots,a_{n}).\]
	
Malheureusement, l'écriture d'un élément de $SL_{2}(\mathbb{Z})$ sous cette forme n'est pas unique (pour une façon d'assurer l'unicité d'une écriture de cette forme on peut consulter \cite{MO} section 6).
\\
\\ \indent Ceci amène naturellement à chercher les différentes écritures d'une matrice, ou d'un ensemble de matrices, du groupe modulaire. On s'intéresse particulièrement au cas de la matrice $Id$. Pour cela, on considère l'équation suivante: \begin{equation}
\label{a}
M_{n}(a_1,\ldots,a_n)=\pm Id.
\end{equation} 

\noindent V.Ovsienko (voir \cite{O} Théorèmes 1 et 2) a entièrement résolu celle-ci sur $\mathbb{N}^{*}$ et donné une description combinatoire des solutions en terme de découpages de polygones. Notons qu'on connaît également une formule donnant le nombre de solutions pour un $n$ fixé (voir \cite{CO2} Théorème 2.2). On dispose aussi des solutions de cette équation sur $\mathbb{N}$ (voir \cite{C} Théorème 3.1), sur $\mathbb{Z}$ (voir \cite{C} Théorème 3.2), et d'une description combinatoire de celles-ci. On peut également la résoudre sur $\mathbb{Z}[\alpha]$ avec $\alpha$ un nombre complexe transcendant (voir \cite{Ma2} Théorème 2.7). La résolution de cette équation sur d'autres ensembles, en particulier des sous-ensembles de $\mathbb{C}$, est encore un problème ouvert (voir \cite{C} problème ouvert 4.1). Notons que les solutions de l'équation $M_{n}(a_{1},\ldots,a_{n})=-Id$ permettent de construire des frises de Coxeter, et, qu'à partir d'une telle frise, on peut obtenir une solution de cette équation (voir \cite{BR}).
\\
\\ \indent On va s'intéresser ici aux cas des anneaux $\mathbb{Z}/N\mathbb{Z}$, c'est-à-dire à l'étude sur $\mathbb{Z}/N\mathbb{Z}$ de l'équation :
\begin{equation}
\label{p}
\tag{$E_{N}$}
M_{n}(a_1,\ldots,a_n)=\pm Id.
\end{equation} On dira, en particulier, qu'une solution de \eqref{p} est de taille $n$ si cette solution est un $n$-uplet d'éléments de $\mathbb{Z}/N\mathbb{Z}$. L'objectif principal de cette étude est de connaître toutes les écritures des éléments des sous-groupes de congruence suivants: \[\hat{\Gamma}(N)=\{A \in SL_{2}(\mathbb{Z})~{\rm tel~que}~A= \pm Id~( {\rm mod}~N)\}\] sous la forme $M_{n}(a_1,\ldots,a_n)$ avec les $a_{i}$ des entiers strictement positifs. 
\\
\\ \indent L'équation \eqref{p} a déjà été étudiée lors de précédents travaux (voir \cite{Ma1, M, M2}). L'élément clef des différents résultats obtenus est l'utilisation d'une notion de solutions irréductibles à partir desquelles on peut construire l'ensemble des solutions (voir section suivante). Grâce à elle, on a pu résoudre complètement \eqref{p} pour $N \leq 6$ (voir \cite{M} section 4). Celle-ci nous a également permis d'obtenir plusieurs résultats généraux d'irréductibilité. 
\\
\\ \indent La majeure partie de ces résultats concernent les solutions monomiales minimales qui sont les solutions de \eqref{p} dont toutes les composantes sont identiques et minimales pour cette propriété (voir \cite{M} section 3.3 et la section suivante). En particulier, on a montré que, pour certaines valeurs de $N$, toutes les solutions monomiales minimales non nulles de \eqref{p} sont irréductibles (voir \cite{M} sections 3.3 et 4 et la section suivante). On cherche ici à étudier les entiers $N$ pour lesquels cette propriété est vraie. Pour cela, on va donner, dans la partie suivante, les définitions qui nous seront utiles pour la suite ainsi qu'un résultat de classification de ces entiers, avant de démontrer celui-ci dans la section \ref{DT}.

\section{Définitions et résultat principal}\label{RP}    

L'objectif de cette section est de fournir les définitions essentielles à l'étude de l'équation \eqref{p}, introduites précédemment dans \cite{C} et \cite{M}, et d'énoncer le résultat principal de ce texte. Sauf mention contraire, $N$ désigne un entier naturel supérieur à $2$, et, s'il n'y a pas d'ambiguïté sur $N$, on note $\overline{a}=a+N\mathbb{Z}$ (avec $a \in \mathbb{Z}$). $\mathbb{P}$ représente l'ensemble des nombres premiers.

\begin{definition}[\cite{C}, lemme 2.7]
\label{21}

Soient $(\overline{a_{1}},\ldots,\overline{a_{n}}) \in (\mathbb{Z}/N \mathbb{Z})^{n}$ et $(\overline{b_{1}},\ldots,\overline{b_{m}}) \in (\mathbb{Z}/N \mathbb{Z})^{m}$. On définit l'opération ci-dessous: \[(\overline{a_{1}},\ldots,\overline{a_{n}}) \oplus (\overline{b_{1}},\ldots,\overline{b_{m}})= (\overline{a_{1}+b_{m}},\overline{a_{2}},\ldots,\overline{a_{n-1}},\overline{a_{n}+b_{1}},\overline{b_{2}},\ldots,\overline{b_{m-1}}).\] Le $(n+m-2)$-uplet obtenu est appelé la somme de $(\overline{a_{1}},\ldots,\overline{a_{n}})$ avec $(\overline{b_{1}},\ldots,\overline{b_{m}})$.

\end{definition}

\begin{examples}

{\rm On donne ci-dessous quelques exemples de sommes :
\begin{itemize}
\item $(\overline{3},\overline{2},\overline{1}) \oplus (\overline{5},\overline{0},\overline{1},\overline{2})= (\overline{5},\overline{2},\overline{6},\overline{0},\overline{1})$;
\item $(\overline{-2},\overline{0},\overline{-1},\overline{1}) \oplus (\overline{-1},\overline{2},\overline{1}) = (\overline{-1},\overline{0},\overline{-1},\overline{0},\overline{2})$;
\item $n \geq 2$, $(\overline{a_{1}},\ldots,\overline{a_{n}}) \oplus (\overline{0},\overline{0}) = (\overline{0},\overline{0}) \oplus (\overline{a_{1}},\ldots,\overline{a_{n}})=(\overline{a_{1}},\ldots,\overline{a_{n}})$.
\end{itemize}
}
\end{examples}

L'opération $\oplus$ définie ci-dessus n'est malheureusement ni commutative ni associative (voir \cite{WZ} exemple 2.1). En revanche, celle-ci possède la propriété suivante : si $(\overline{b_{1}},\ldots,\overline{b_{m}})$ est une solution de \eqref{p} alors la somme $(\overline{a_{1}},\ldots,\overline{a_{n}}) \oplus (\overline{b_{1}},\ldots,\overline{b_{m}})$ est solution de \eqref{p} si et seulement si $(\overline{a_{1}},\ldots,\overline{a_{n}})$ est solution de \eqref{p} (voir \cite{C,WZ} et \cite{M} proposition 3.7). 

\begin{definition}[\cite{C}, définition 2.5]
\label{22}

 Soient $(\overline{a_{1}},\ldots,\overline{a_{n}}) \in (\mathbb{Z}/N \mathbb{Z})^{n}$ et $(\overline{b_{1}},\ldots,\overline{b_{n}}) \in (\mathbb{Z}/N \mathbb{Z})^{n}$. On dit que $(\overline{a_{1}},\ldots,\overline{a_{n}}) \sim (\overline{b_{1}},\ldots,\overline{b_{n}})$ si $(\overline{b_{1}},\ldots,\overline{b_{n}})$ est obtenu par permutation circulaire de $(\overline{a_{1}},\ldots,\overline{a_{n}})$ ou de $(\overline{a_{n}},\ldots,\overline{a_{1}})$.

\end{definition}

On vérifie aisément que $\sim$ est une relation d'équivalence sur les $n$-uplets d'éléments de $\mathbb{Z}/N \mathbb{Z}$ (voir \cite{WZ} lemme 1.7). D'autre part, si $(\overline{a_{1}},\ldots,\overline{a_{n}}) \sim (\overline{b_{1}},\ldots,\overline{b_{n}})$ alors $(\overline{a_{1}},\ldots,\overline{a_{n}})$ est solution de \eqref{p} si et seulement si $(\overline{b_{1}},\ldots,\overline{b_{n}})$ est solution de \eqref{p} (voir \cite{C} proposition 2.6).

\begin{definition}[\cite{C}, définition 2.9]
\label{23}

Une solution $(\overline{c_{1}},\ldots,\overline{c_{n}})$ avec $n \geq 3$ de \eqref{p} est dite réductible s'il existe une solution de \eqref{p} $(\overline{b_{1}},\ldots,\overline{b_{l}})$ et un $m$-uplet $(\overline{a_{1}},\ldots,\overline{a_{m}})$ d'éléments de $\mathbb{Z}/N \mathbb{Z}$ tels que \begin{itemize}
\item $(\overline{c_{1}},\ldots,\overline{c_{n}}) \sim (\overline{a_{1}},\ldots,\overline{a_{m}}) \oplus (\overline{b_{1}},\ldots,\overline{b_{l}})$;
\item $m \geq 3$ et $l \geq 3$.
\end{itemize}
Une solution est dite irréductible si elle n'est pas réductible.

\end{definition}

\begin{remark} 

{\rm On ne considère pas $(\overline{0},\overline{0})$ comme une solution irréductible de \eqref{p}.}

\end{remark}

\indent Au cours de l'étude de l'équation \eqref{p}, on a introduit la notion de solutions monomiales rappelée dans la définition qui suit :

\begin{definition}[\cite{M}, définition 3.9]
\label{24}

i)~Soient $n \in \mathbb{N}^{*}$ et $\overline{k} \in \mathbb{Z}/N\mathbb{Z}$. On appelle solution $(n,\overline{k})$-monomiale un $n$-uplet d'éléments de $\mathbb{Z}/ N \mathbb{Z}$ constitué uniquement de $\overline{k}$ et solution de \eqref{p}.
\\
\\ ii)~On appelle solution monomiale une solution pour laquelle il existe $m \in \mathbb{N}^{*}$ et $\overline{l} \in \mathbb{Z}/N\mathbb{Z}$ tels qu'elle est $(m,\overline{l})$-monomiale.
\\
\\ iii)~On appelle solution $\overline{k}$-monomiale minimale une solution $(n,\overline{k})$-monomiale avec $n$ le plus petit entier pour lequel il existe une solution $(n,\overline{k})$-monomiale.
\\
\\ iv)~On appelle solution monomiale minimale une solution $\overline{k}$-monomiale minimale pour un $\overline{k} \in \mathbb{Z}/N\mathbb{Z}$.

\end{definition}

On a déjà démontré un certain nombre de propriétés d'irréductibilité pour ces solutions (voir \cite{M, M2} et la section \ref{EMI}). En particulier, on a

\begin{theorem}[\cite{M}, Théorème 3.16]
\label{25}

Si $N$ est premier alors les solutions monomiales minimales non nulles de \eqref{p} sont irréductibles. 

\end{theorem}

Cependant, l'étude de \eqref{p} pour les petites valeurs de $N$ montrent que cette propriété est encore vérifiée pour $N=4$ et $N=6$ (voir \cite{M} Théorème 2.5). Cela amène naturellement au problème de la classification des entiers naturels $N$ pour lesquels les solutions monomiales minimales non nulles de \eqref{p} sont irréductibles. On introduit, pour cela, la notion suivante : 
		
\begin{definition}
\label{26}

Soit $N \geq 2$. On dit que $N$ est monomialement irréductible si les solutions monomiales minimales non nulles de \eqref{p} sont irréductibles. Dans le cas contraire, on dit que $N$ est monomialement réductible.

\end{definition}

On cherche donc à classifier les entiers monomialement irréductibles, ou, de façon plus modeste, à avoir des informations sur ces-derniers. Pour cela, on pose \[\Omega=\{107, 163, 173, 277, 283, 317, 347, 523, 557, 563, 613, 653\}\] \[\cup~\{733, 773, 787, 877, 907, 997\}.\]

\begin{theorem}
\label{27}

Soit $N \geq 2$. 
\\
\\i) Si $N$ est premier ou si $N \in \{4, 6, 8, 12, 24\}$ alors $N$ est monomialement irréductible.
\\
\\ii) On suppose $N$ pair. $N$ est monomialemnt irréductible si et seulement si $N \in \{2, 4, 6, 8, 12, 24\}$.
\\
\\iii) On suppose $N$ impair non premier. Si $N$ est monomialement irréductible alors $N$ est de la forme $\prod_{i=1}^{r} p_{i}$ où $r \geq 2$ et les $p_{i}$ sont des nombres premiers impairs deux à deux distincts vérifiant les conditions suivantes :
\begin{itemize}
\item $\forall i \in [\![1;r]\!]$, $p_{i} \equiv \pm 3 [5]$; 
\item $\forall i \in [\![1;r]\!]$, $p_{i} \equiv \pm 3 [8]$;
\item si $p_{i} \leq 1000$ alors $p_{i} \in \Omega$.
\end{itemize}

\end{theorem}

Ce théorème, démontré dans la section suivante, permet d'obtenir une description exhaustive des entiers monomialement irréductibles pairs et donne des propriétés assez précises sur les entiers monomialement irréductibles impairs. Un certain nombre d'éléments permettant de mettre en relief ce résultat sont donnés dans la section suivante. De plus, des résultats obtenus informatiquement, résumés dans la section \ref{CC}, permettent d'avoir une vision encore plus fine des entiers monomialement irréductibles impairs. 

\section{Démonstration du théorème \ref{27}}
\label{DT}

L'objectif de cette section est de fournir les différents éléments nécessaires à la preuve du théorème \ref{27}.

\subsection{Résultats préliminaires}
\label{RP}

On commence par donner un certain nombre de résultats classiques nécessaires pour la suite. 

\begin{theorem}[Lemme chinois; \cite{R}, Théorème 10.10]
\label{31}

Soit $(n,m) \in (\mathbb{N}^{*})^{2}$ avec $m$ et $n$ premiers entre eux. $\exists (u,v) \in \mathbb{Z}^{2}$ tel que $un+vm=1$. L'application \[\begin{array}{ccccc} 
\chi & : & \mathbb{Z}/mn\mathbb{Z} & \longrightarrow & \mathbb{Z}/m\mathbb{Z} \times \mathbb{Z}/n\mathbb{Z} \\
 & & x+mn\mathbb{Z} & \longmapsto & (x+m\mathbb{Z},x+n\mathbb{Z})  \\
\end{array}\] est un isomorphisme d'anneaux dont la bijection réciproque est \[\begin{array}{ccccc} 
\psi & : & \mathbb{Z}/m\mathbb{Z} \times \mathbb{Z}/n\mathbb{Z} & \longrightarrow & \mathbb{Z}/mn\mathbb{Z} \\
 & & (x+m\mathbb{Z},y+n\mathbb{Z}) & \longmapsto & (xun+yvm)+mn\mathbb{Z}  \\
\end{array}.\]

\end{theorem}

On aura également besoin de quelques éléments sur les carrés modulo $N$. Si $p$ est un nombre premier impair et si $a$ est un entier premier avec $p$, on note $\left(\dfrac{a}{p}\right)$ le symbole de Legendre c'est-à-dire :
\[\left(\dfrac{a}{p}\right)=\left\{
    \begin{array}{ll}
        1 & \mbox{si } \overline{a}~{\rm est~un~carr\acute{e}~dans}~\mathbb{Z}/p\mathbb{Z}; \\
        -1 & \mbox{sinon }.
    \end{array}
\right.  \\ \]

\noindent On dispose du résultat suivant :

\begin{theorem}
\label{32}

i) (Loi de réciprocité quadratique de Gauss; \cite{G}, Théorème XII.25) 
\\Soient $p$ et $q$ deux nombres premiers impairs distincts. On a \[\left(\dfrac{p}{q}\right)\left(\dfrac{q}{p}\right)= (-1)^{\frac{p-1}{2}\frac{q-1}{2}}.\]
\\
\\ ii) (Loi complémentaire; \cite{G}, proposition XII.27) 
\\Soit $p$ un nombre premier impair, $\left(\dfrac{2}{p}\right)= (-1)^{\frac{p^{2}-1}{8}}$.

\end{theorem}

\noindent Ce résultat permet de démontrer l'énoncé ci-dessous :

\begin{lemma}
\label{33}

Soit $p$ un nombre premier impair.
\\
\\ i) 5 est un carré modulo $p$ si et seulement si $p=5$ ou $p \equiv \pm 1 [5]$.
\\
\\ ii) 2 est un carré modulo $p$ si et seulement si $p \equiv \pm 1 [8]$.

\end{lemma}

\begin{proof}

i) 5 est un carré modulo 5. Supposons maintenant $p \neq 5$. D'après le théorème précédent,
\[\left(\dfrac{5}{p}\right)=\left(\dfrac{p}{5}\right) (-1)^{\frac{p-1}{2}\frac{5-1}{2}}=\left(\dfrac{p}{5}\right) (-1)^{p-1}=\left(\dfrac{p}{5}\right)~{\rm (car}~p-1~{\rm est~pair)}.\]
Donc, 5 est un carré modulo $p$ si et seulement si $p$ est un carré modulo 5 c'est-à-dire si et seulement si $p \equiv \pm 1 [5]$ ou $p=5$.
\\
\\ii) Si $p \equiv \pm 1 [8]$. Il existe $k \in \mathbb{Z}$ tel que $p=8k \pm 1$. On a \[\left(\dfrac{2}{p}\right)= (-1)^{\frac{p^{2}-1}{8}}=(-1)^{\frac{64k^{2} \pm 16k}{8}}=((-1)^{4k^{2} \pm k})^{2}=1.\]
\\Si $p \not\equiv \pm 1 [8]$ alors comme $p$ est impair $p \equiv \pm 3 [8]$. Il existe $k \in \mathbb{Z}$ tel que $p=8k \pm 3$. On a \[\left(\dfrac{2}{p}\right)= (-1)^{\frac{p^{2}-1}{8}}=(-1)^{\frac{64k^{2} \pm 48k+8}{8}}=((-1)^{4k^{2} \pm 3k})^{2}(-1)^{1}=-1.\]

\end{proof}

\noindent On aura également besoin dans la suite du résultat suivant sur l'expression de la matrice $M_{n}(a_{1},\ldots,a_{n})$ en terme de déterminant. On pose $K_{-1}=0$, $K_{0}=1$ et on note pour $i \geq 1$ \[K_i(a_{1},\ldots,a_{i})=
\left|
\begin{array}{cccccc}
a_1&1&&&\\[4pt]
1&a_{2}&1&&\\[4pt]
&\ddots&\ddots&\!\!\ddots&\\[4pt]
&&1&a_{i-1}&\!\!\!\!\!1\\[4pt]
&&&\!\!\!\!\!1&\!\!\!\!a_{i}
\end{array}
\right|.\] $K_{i}(a_{1},\ldots,a_{i})$ est le continuant de $a_{1},\ldots,a_{i}$. On dispose de l'égalité suivante (voir \cite{CO,MO}) : \[M_{n}(a_{1},\ldots,a_{n})=\begin{pmatrix}
    K_{n}(a_{1},\ldots,a_{n}) & -K_{n-1}(a_{2},\ldots,a_{n}) \\
    K_{n-1}(a_{1},\ldots,a_{n-1})  & -K_{n-2}(a_{2},\ldots,a_{n-1}) 
   \end{pmatrix}.\]
	
\begin{lemma}
\label{34}

Soient $n$ et $N$ deux entiers naturels supérieurs à 2 et $\overline{k} \in \mathbb{Z}/N\mathbb{Z}$ tels que $K_{n}(\overline{k},\ldots,\overline{k})=\overline{\epsilon}$ avec $\epsilon=\pm 1$. Posons $\overline{a}=\overline{\epsilon}K_{n-1}(\overline{k},\ldots,\overline{k})$. $(\overline{a},\overline{k},\ldots,\overline{k},\overline{a}) \in (\mathbb{Z}/N\mathbb{Z})^{n+2}$ est une solution de \eqref{p}.

\end{lemma}

\begin{proof}

\begin{eqnarray*}
M &=& M_{n+2}(\overline{a},\overline{k},\ldots,\overline{k},\overline{a}) \\
  &=& \begin{pmatrix}
    K_{n+2}(\overline{a},\overline{k},\ldots,\overline{k},\overline{a}) & -K_{n+1}(\overline{k},\ldots,\overline{k},\overline{a}) \\
    K_{n+1}(\overline{a},\overline{k},\ldots,\overline{k})  & -K_{n}(\overline{k},\ldots,\overline{k}) 
   \end{pmatrix} \\
	&=& \begin{pmatrix}
    K_{n+2}(\overline{a},\overline{k},\ldots,\overline{k},\overline{a}) & -K_{n+1}(\overline{k},\ldots,\overline{k},\overline{a}) \\
    K_{n+1}(\overline{a},\overline{k},\ldots,\overline{k})  & -\overline{\epsilon} 
   \end{pmatrix}. \\
\end{eqnarray*}

\noindent Comme $\overline{\epsilon}^{2}=\overline{1}$, $K_{n-1}(\overline{k},\ldots,\overline{k})=\overline{a\epsilon}$.
\\
\\On a, $K_{n+1}(\overline{k},\ldots,\overline{k},\overline{a})=\overline{a}K_{n}(\overline{k},\ldots,\overline{k})-K_{n-1}(\overline{k},\ldots,\overline{k})=\overline{a\epsilon}-\overline{a\epsilon}=\overline{0}$. On procède de la même façon pour le calcul de $K_{n+1}(\overline{a},\overline{k},\ldots,\overline{k})$. 
\\
\\Comme $M \in SL_{2}(\mathbb{Z}/N\mathbb{Z})$, $K_{n+2}(\overline{a},\overline{k},\ldots,\overline{k},\overline{a})=-\overline{\epsilon}$. Ainsi, $M=-\overline{\epsilon} Id$.

\end{proof}

\begin{proposition}
\label{spe}

Soit $n$ un entier naturel supérieur à 2. $(\overline{a_{1}},\ldots,\overline{a_{n}})$ est une solution irréductible de \eqref{p} si et seulement si $(\overline{-a_{1}},\ldots,\overline{-a_{n}})$ est une solution irréductible de \eqref{p}.

\end{proposition}

\begin{proof}

Soit $(\overline{a_{1}},\ldots,\overline{a_{n}})$ une solution irréductible de \eqref{p}. $(\overline{-a_{1}},\ldots,\overline{-a_{n}})$ est une solution de \eqref{p} (voir \cite{M} proposition 3.6 ii)). Supposons par l'absurde que $(\overline{-a_{1}},\ldots,\overline{-a_{n}})$ est réductible. 
\\
\\Il existe $(\overline{b_{1}},\ldots,\overline{b_{l}})$ et $(\overline{c_{1}},\ldots,\overline{c_{l'}})$ solutions de \eqref{p} différentes de $(\overline{0},\overline{0})$ avec $l+l'=n+2$ et $l,l' \geq 3$ telles que

\begin{eqnarray*}
(\overline{-a_{1}},\ldots,\overline{-a_{n}}) &=& (\overline{b_{1}},\ldots,\overline{b_{l}}) \oplus (\overline{c_{1}},\ldots,\overline{c_{l'}}) \\
                                             &=& (\overline{b_{1}+c_{l'}},\overline{b_{2}},\ldots,\overline{b_{l-1}},\overline{b_{l}+c_{1}},\overline{c_{2}},\ldots,\overline{c_{l'-1}}). \\
\end{eqnarray*}

\noindent Donc, on a 

\begin{eqnarray*}
(\overline{a_{1}},\ldots,\overline{a_{n}}) &=& (\overline{-b_{1}-c_{l'}},\overline{-b_{2}},\ldots,\overline{-b_{l-1}},\overline{-b_{l}-c_{1}},\overline{-c_{2}},\ldots,\overline{-c_{l'-1}}) \\
                                           &=& (\overline{-b_{1}},\ldots,\overline{-b_{l}}) \oplus (\overline{-c_{1}},\ldots,\overline{-c_{l'}}).\\
\end{eqnarray*}

\noindent Comme $(\overline{-b_{1}},\ldots,\overline{-b_{l}})$ et $(\overline{-c_{1}},\ldots,\overline{-c_{l'}})$ sont des solutions de \eqref{p} (voir \cite{M} proposition 3.6 ii)), $(\overline{a_{1}},\ldots,\overline{a_{n}})$ est réductible. Ceci est absurde.
\\
\\Ainsi, $(\overline{-a_{1}},\ldots,\overline{-a_{n}})$ est une solution irréductible de \eqref{p}. On procède, de même, pour démontrer l'autre implication.

\end{proof}

\subsection{Quelques entiers monomialement irréductibles}
\label{EMI}

On souhaite chercher d'autres entiers monomialement irréductibles non premiers. Pour montrer qu'un entier est monomialement irréductible il faut vérifier que toutes les solutions monomiales minimales non nulles sont irréductibles. 
\\
\\On dispose déjà des résultats généraux suivants (voir \cite{M,M2}) :

\begin{itemize}
\item les solutions $\pm \overline{1}$-monomiales minimales de \eqref{p} sont irréductibles et ce sont les seules solutions de taille 3;
\item les solutions de \eqref{p} de taille 4 sont de la forme $(\overline{a},\overline{b},\overline{a},\overline{b})$ avec $\overline{ab}=\overline{2}$ et $(\overline{-a},\overline{b},\overline{a},\overline{-b})$ avec $\overline{ab}=\overline{0}$;
\item si $N \geq 3$, les solutions $\pm \overline{2}$-monomiales minimales de \eqref{p} sont irréductibles;
\item si $N$ est pair, la solution $\overline{\frac{N}{2}}$-monomiale minimale de \eqref{p} est irréductible.
\\
\end{itemize}

\noindent Notons que ces résultats nous permettent de retrouver immédiatement que 4 et 6 sont monomialement irréductibles.
\\
\\Par la proposition \ref{spe}, la solution $\overline{k}$-monomiale minimale de \eqref{p} est irréductible si et seulement si la solution $\overline{-k}$-monomiale minimale de \eqref{p} est irréductible.
\\
\\D'autre part, si $(\overline{a},\overline{k},\ldots,\overline{k},\overline{b})$ est une solution de \eqref{p} alors $\overline{a}=\overline{b}$ et $\overline{a}(\overline{a}-\overline{k})=\overline{0}$ (voir \cite{M} proposition 3.15). En particulier, si $\overline{k} \neq \overline{0}$ et si ce polynôme a pour seules racines $\overline{0}$ et $\overline{k}$ alors la solution $\overline{k}$- monomiale minimale de \eqref{p} est irréductible. La réciproque est malheureusement fausse comme on va pouvoir le constater dans ce qui suit. 

\begin{proposition}
\label{35}

8, 12 et 24 sont monomialement irréductibles.

\end{proposition}

\begin{proof}

Pour $N=8$, le polynôme $X(X-\overline{3})=\overline{0}$ a pour seules racines $\overline{0}$ et $\overline{3}$. Ainsi, les solutions monomiales minimales non nulles de $(E_{8})$ sont irréductibles.
\\
\\Pour $N=12$, on distingue les 3 cas :
\begin{itemize}
\item Le polynôme $X(X-\overline{3})=\overline{0}$ a pour seules racines $\overline{0}$ et $\overline{3}$. Ainsi, la solution $\overline{3}$-monomiale minimale de $(E_{12})$ est irréductible. 
\item Le polynôme $X(X-\overline{4})=\overline{0}$ a pour racines $\overline{0}$, $\overline{4}$, $\overline{6}$, $\overline{10}$. La solution $\overline{4}$-monomiale minimale de $(E_{12})$ est de taille 12. Il n'existe pas de solution de la forme $(\overline{6},\overline{4},\ldots,\overline{4},\overline{6})$ de taille inférieure à 10. Donc, la solution $\overline{4}$-monomiale minimale de $(E_{12})$ est irréductible. 
\item La solution $\overline{5}$-monomiale minimale de $(E_{12})$ est de taille 6. $\overline{5}^{2}=\overline{1} \neq \overline{0}$ et $\overline{5} \neq \pm \overline{1}$, on ne peut donc pas réduire cette solution avec une solution de taille 3 ou 4. Ainsi, la solution $\overline{5}$-monomiale minimale de $(E_{12})$ est irréductible. 
\\
\end{itemize} 

\noindent Pour $N=24$, on regarde pour $k \in [\![3;11]\!]$, la taille de la solution $\overline{k}$-monomiale minimale de $(E_{24})$ ainsi que les racines de $X(X-\overline{k})$ différentes de $\overline{0}$ et $\overline{k}$.

\begin{center}
\begin{tabular}{|c|c|c|c|c|c|c|c|m{1cm}|c|}
  \hline
  $\overline{k}$  & $\overline{3}$ & $\overline{4}$ & $\overline{5}$ & $\overline{6}$ & $\overline{7}$  & $\overline{8}$ & $\overline{9}$ & \centering $\overline{10}$ & $\overline{11}$ \rule[-7pt]{0pt}{20pt} \\
  \hline
  taille  & 12 & 12 & 6 & 8 & 6 & 12 & 12 & \centering 24 & 6  \rule[-7pt]{0pt}{20pt} \\
  \hline
	racines & \diagbox[width=10mm,height=22mm]{}{} & $\overline{12}$, $\overline{16}$ & $\overline{8}$, $\overline{21}$ & $\overline{12}$, $\overline{18}$ & $\overline{15}$, $\overline{16}$ & $\overline{12}$, $\overline{20}$ & \diagbox[width=10mm,height=22mm]{}{} & \centering $\overline{4}$, $\overline{6}$ $\overline{12}$, $\overline{22}$ $\overline{16}$, $\overline{18}$ & $\overline{3}$, $\overline{8}$  \rule[-7pt]{0pt}{20pt} \\
  \hline
\end{tabular}
\end{center}

\noindent Les solutions $\overline{3}$ et $\overline{9}$ monomiales minimales de $(E_{24})$ sont irréductibles. Pour les autres, on calcule $M_{l}(\overline{a},\overline{k},\ldots,\overline{k},\overline{a})$, avec, $\overline{a}$ pris dans les éléments de la ligne racines et de la colonne débutant par $\overline{k}$, et, $l$ variant entre 4 et la taille de la solution $\overline{k}$-monomiale minimale moins 2.
\\
\\En effectuant les calculs, on ne trouve aucune solution de la forme souhaitée. Les solutions monomiales minimales non nulles de $(E_{24})$ sont donc irréductibles.

\end{proof}

\noindent Ceci prouve le premier point du théorème \ref{27}.

\subsection{Le cas des facteurs carrés}
\label{FC}

On va maintenant traiter le cas des entiers possédant un facteur carré.

\begin{proposition}[\cite{M2}, propositions 3.5, 3.6 et 3.8]
\label{36}

Soit $p$ un nombre premier. Si $p^{2} \mid N$ alors la solution $\overline{\frac{N}{p}}$-monomiale minimale de \eqref{p} est de taille $2p$. Celle-ci est irréductible si et seulement si $p=2$.

\end{proposition}

\begin{proposition}
\label{37}

Si $16 \mid N$ alors la solution $\overline{\frac{N}{4}}$-monomiale minimale de \eqref{p} est réductible et de taille $8$. 

\end{proposition}

\begin{proof}

Par la proposition 3.5 de \cite{M2}, $(\overline{\frac{N}{4}},\ldots,\overline{\frac{N}{4}}) \in (\mathbb{Z}/N\mathbb{Z})^{8}$ est solution de \eqref{p}. Donc, la taille de la solution $\overline{\frac{N}{4}}$-monomiale minimale de \eqref{p} divise 8, c'est-à-dire qu'elle est égale à 1, 2, 4 ou 8.
\\
\\\eqref{p} n'a pas de solution de taille 1 et $\overline{\frac{N}{4}} \neq \overline{0}$, donc la taille est égale à 4 ou 8. De plus, $\overline{\frac{N}{4}}\overline{\frac{N}{4}}=\overline{N\frac{N}{16}}=\overline{0}$ et $-\overline{\frac{N}{4}} \neq \overline{\frac{N}{4}}$ donc la taille est différente de 4. On en déduit que la solution $\overline{\frac{N}{4}}$-monomiale minimale de \eqref{p} est de taille $8$. 
\\
\\Puisque $\overline{\frac{N}{4}}\overline{\frac{N}{4}}=\overline{0}$, $(\overline{-\frac{N}{4}},\overline{\frac{N}{4}},\overline{\frac{N}{4}},\overline{-\frac{N}{4}})$ est une solution de \eqref{p}. De plus, \[(\overline{\frac{N}{4}},\overline{\frac{N}{4}},\overline{\frac{N}{4}},\overline{\frac{N}{4}},\overline{\frac{N}{4}},\overline{\frac{N}{4}},\overline{\frac{N}{4}},\overline{\frac{N}{4}})=(\overline{\frac{N}{2}},\overline{\frac{N}{4}},\overline{\frac{N}{4}},\overline{\frac{N}{4}},\overline{\frac{N}{4}},\overline{\frac{N}{2}}) \oplus (\overline{-\frac{N}{4}},\overline{\frac{N}{4}},\overline{\frac{N}{4}},\overline{-\frac{N}{4}}).\] 
\noindent Ainsi, la solution $\overline{\frac{N}{4}}$-monomiale minimale de \eqref{p} est réductible.

\end{proof}

\noindent On en déduit qu'un entier monomialement irréductible est de la forme $2^{a} \prod_{i=1}^{r} p_{i}$ où $a \in [\![0;3]\!]$, $r \geq 0$ et les $p_{i}$ sont des nombres premiers impairs deux à deux distincts (avec la convention $\prod_{i=1}^{0} p_{i}=1$).

\subsection{Quelques éléments sur les nombres impairs}
\label{NI}

Jusqu'à présent, on a cherché des valeurs de $\overline{k}$ particulières pour lesquelles on a calculé la taille et cherché une solution pour la réduire. On va ici procéder dans l'autre sens en imposant la taille de la solution utilisée pour réduire et en cherchant une valeur de $\overline{k}$ pour laquelle la solution $\overline{k}$-monomiale minimale sera réduite par une solution de la taille choisie.
\\
\\On commence par le résultat ci-dessous :

\begin{proposition}
\label{38}

Soit $N \geq 2$ non premier et $p$ un nombre premier divisant $N$. Si $p=5$ ou si $p \equiv \pm 1[5]$ alors $N$ est monomialement réductible.

\end{proposition}

\begin{proof}

Si $p^{2} \mid N$ alors, par la proposition \ref{36}, $N$ est monomialement réductible. On suppose donc que $p^{2}$ ne divise pas $N$. Ainsi, il existe un entier $m \geq 2$ tel que $N=mp$ avec $m$ et $p$ premiers entre eux. 
\\
\\Montrons que la condition de l'énoncé suffit à avoir l'existence d'un $\overline{k}$ pour lequel la solution $\overline{k}$-monomiale minimale peut être réduite à l'aide d'une solution de taille 5. On a \[M_{5}(\overline{a},\overline{k},\overline{k},\overline{k},\overline{a})=\begin{pmatrix}
    K_{5}(\overline{a},\overline{k},\overline{k},\overline{k},\overline{a}) & -K_{4}(\overline{k},\overline{k},\overline{k},\overline{a}) \\
    K_{4}(\overline{a},\overline{k},\overline{k},\overline{k})  & -K_{3}(\overline{k},\overline{k},\overline{k}) 
   \end{pmatrix}.\]
\noindent En particulier, on a \[K_{3}(\overline{k},\overline{k},\overline{k})-\overline{1}=\overline{k}^{3}-2\overline{k}-\overline{1}=(\overline{k}+\overline{1})(\overline{k}^{2}-\overline{k}-\overline{1}).\] Le discriminant de $\overline{k}^{2}-\overline{k}-\overline{1}$ est $\overline{5}$. Ainsi, si $p=5$ ou si $p \equiv \pm 1[5]$, $X^{2}-X-1$ a deux racines modulo $p$ (voir lemme \ref{33}). Soit $x+p\mathbb{Z}$ une racine de ce polynôme.
\\
\\On pose $\overline{l}=\psi(-1+m\mathbb{Z}, x+p\mathbb{Z})$ ($m$ et $p$ premiers entre eux). Comme $\psi$ est un morphisme d'anneaux unitaires, on a \[K_{3}(\overline{l},\overline{l},\overline{l})=\overline{1}.\] On pose \[\overline{a}=K_{2}(\overline{l},\overline{l})=\overline{l}^{2}-\overline{1}.\] Par le lemme \ref{34}, $(\overline{a},\overline{l},\overline{l},\overline{l},\overline{a})$ est une solution de \eqref{p}.
\\
\\Il nous reste maintenant à prouver que la solution $\overline{l}$-monomiale minimale de \eqref{p} est de taille supérieure à 6. Notons $h$ cette taille. 
\\
\\Comme les $K_{i}(\overline{l},\ldots,\overline{l})$ sont des polynômes en $\overline{l}$, on a, par le lemme chinois, que la taille de la solution $(-1+m\mathbb{Z})$-monomiale minimale de $(E_{m})$ divise $h$ et que la taille de la solution $(x+p\mathbb{Z})$-monomiale minimale de $(E_{p})$ divise $h$. Comme la	taille d'une solution $(-1+m\mathbb{Z})$-monomiale est toujours un multiple de 3, $h$ est un multiple de 3. De plus, $h \neq 3$. En effet, si $h=3$ alors la solution $(x+p\mathbb{Z})$-monomiale minimale de $(E_{p})$ est de taille 3. Or, ceci est impossible car $x+p\mathbb{Z}\neq \pm 1 +p\mathbb{Z}$ (puisque $\pm 1$ n'est pas une racine modulo $p$ de $X^{2}-X-1$).
\\
\\Ainsi, $h \geq 6$. Donc, la solution $\overline{l}$-monomiale minimale de \eqref{p} est réductible et $N$ est monomialement réductible.

\end{proof}

En utilisant des méthodes similaires, on peut également obtenir d'autres propriétés des entiers monomialement irréductibles.

\begin{proposition}
\label{39}

Soit $N \geq 2$ non premier et $p$ un nombre premier divisant $N$. Si $p \equiv \pm 1[8]$ alors $N$ est monomialement réductible.

\end{proposition}

\begin{proof}

Si $p^{2} \mid N$ alors, par la proposition \ref{36}, $N$ est monomialement réductible. On suppose donc que $p^{2}$ ne divise pas $N$. Ainsi, il existe un entier $m \geq 2$ tel que $N=mp$ avec $m$ et $p$ premiers entre eux. 
\\
\\Montrons que la condition de l'énoncé suffit à avoir l'existence d'un $\overline{k}$ pour lequel la solution $\overline{k}$-monomiale minimale peut être réduite à l'aide d'une solution de taille 6. On a \[M_{6}(\overline{a},\overline{k},\overline{k},\overline{k},\overline{k},\overline{a})=\begin{pmatrix}
    K_{6}(\overline{a},\overline{k},\overline{k},\overline{k},\overline{k},\overline{a}) & -K_{5}(\overline{k},\overline{k},\overline{k},\overline{k},\overline{a}) \\
    K_{5}(\overline{a},\overline{k},\overline{k},\overline{k},\overline{k})  & -K_{4}(\overline{k},\overline{k},\overline{k},\overline{k}) 
   \end{pmatrix}.\]
\noindent En particulier, on a \[K_{4}(\overline{k},\overline{k},\overline{k},\overline{k})+\overline{1}=\overline{k}^{4}-3\overline{k^{2}}+\overline{1}+\overline{1}=(\overline{k}+\overline{1})(\overline{k}-\overline{1})(\overline{k}^{2}-\overline{2}).\]  Si $p \equiv \pm 1[8]$, $X^{2}-2$ a deux racines modulo $p$ (voir lemme \ref{33}). Soit $y+p\mathbb{Z}$ une racine de ce polynôme.
\\
\\On pose $\overline{l}=\psi(1+m\mathbb{Z}, y+p\mathbb{Z})$ ($m$ et $p$ premiers entre eux). Comme $\psi$ est un morphisme d'anneaux unitaires, on a \[K_{4}(\overline{l},\overline{l},\overline{l},\overline{l})=-\overline{1}.\] On pose \[\overline{a}=-K_{3}(\overline{l},\overline{l},\overline{l})=-\overline{l}^{3}+\overline{2l}.\] Par le lemme \ref{34}, $(\overline{a},\overline{l},\overline{l},\overline{l},\overline{l},\overline{a})$ est une solution de \eqref{p}.
\\
\\Il nous reste maintenant à prouver que la solution $\overline{l}$-monomiale minimale de \eqref{p} est de taille supérieure à 7. Notons $h$ cette taille. 
\\
\\Comme les $K_{i}(\overline{l},\ldots,\overline{l})$ sont des polynômes en $\overline{l}$, on a, par le lemme chinois, que la taille de la solution $(1+m\mathbb{Z})$-monomiale minimale de $(E_{m})$ divise $h$ et que la taille de la solution $(y+p\mathbb{Z})$-monomiale minimale de $(E_{p})$ divise $h$. Comme la solution $(1+m\mathbb{Z})$-monomiale minimale de $(E_{m})$ est de taille 3, $h$ est un multiple de 3. 
\\
\\$h \neq 3$. En effet, si $h=3$ alors la solution $(y+p\mathbb{Z})$-monomiale minimale de $(E_{p})$ est de taille 3. Or, ceci est impossible car $y+p\mathbb{Z}\neq \pm 1 +p\mathbb{Z}$ (puisque $\pm 1$ n'est pas une racine modulo $p$ de $X^{2}-2$).
\\
\\$h \neq 6$. En effet, si $h=6$ alors $M_{6}(\overline{l},\overline{l},\overline{l},\overline{l},\overline{l},\overline{l})=Id$ (puisque $-K_{4}(\overline{l},\overline{l},\overline{l},\overline{l})=\overline{1}$). Donc, par le lemme chinois, \[K=K_{6}(y+p\mathbb{Z},\ldots,y+p\mathbb{Z})=1+p\mathbb{Z}.\] Or, on a :
\begin{eqnarray*}
K &=& (y+p\mathbb{Z})^{6}-5(y+p\mathbb{Z})^{4}+6(y+p\mathbb{Z})^{2}-1+p\mathbb{Z} \\
  &=& (8+p\mathbb{Z})-5(4+p\mathbb{Z})+6(2+p\mathbb{Z})-1+p\mathbb{Z}~{\rm(car~(y+p\mathbb{Z})^{2}=2+p\mathbb{Z})}\\
	&=&-1+p\mathbb{Z}.\\
\end{eqnarray*}
\noindent Ainsi, $h \geq 7$. Donc, la solution $\overline{l}$-monomiale minimale de \eqref{p} est réductible et $N$ est monomialement réductible.

\end{proof}

\begin{remarks}

{\rm i) Par le théorème faible de la progression arithmétique de Dirichlet (voir \cite{G} proposition VII.13), il existe une infinité de nombres premiers congrus à $1$ modulo 5 et une infinité de nombres premiers congrus à $1$ modulo 8.
\\
\\ii) Les démonstration des deux propositions précédentes sont constructives.}

\end{remarks}

\noindent On donne ci-dessous un exemple d'application pour chacune des deux situations.

\begin{examples}

{\rm i) Si $N=40=8 \times 5$. On a la relation de Bézout suivante : $1=16-15=8\times 2+5 \times(-3)$. 3 est une racine modulo 5 de $X^{2}-X-1$. On pose $l=16\times 3-15\times (-1)=63$. Par la proposition \ref{38}, la solution $\overline{23}$-monomiale minimale de $(E_{40})$ est réductible. $23^{2}-1=528=40\times 13+8$, donc, on peut la réduire avec la solution $(\overline{8},\overline{23},\overline{23},\overline{23},\overline{8})$. Si on effectue les calculs, on trouve que la taille de la solution $\overline{23}$-monomiale minimale de $(E_{40})$ est égale à 30.
\\
\\ii) Si $N=56=7 \times 8$. On a la relation de Bézout suivante : $1=8-7$. 3 est une racine modulo 7 de $X^{2}-2$. On pose $l=8\times 3-7\times 1=17$. Par la proposition \ref{39}, la solution $\overline{17}$-monomiale minimale de $(E_{56})$ est réductible. $-(17^{3}-2\times 17)+56\mathbb{Z}=49+56\mathbb{Z}$, donc, on peut la réduire avec la solution $(\overline{49},\overline{17},\overline{17},\overline{17},\overline{17},\overline{49})$. Si on effectue les calculs, on trouve que la taille de la solution $\overline{17}$-monomiale minimale de $(E_{56})$ est égale à 24.}

\end{examples}

Les deux résultats précédents permettent donc d'envisager l'étude d'une infinité de nombres premiers. Cela dit, les nombres ainsi étudiés représentent-ils "beaucoup" de nombres premiers ? En d'autres termes, peut-on estimer la quantité suivante lorsque $x$ tend vers l'infini : \[D(x):=\frac{{\rm card}(\{p \in \mathbb{P}\cap [0,x],~p \equiv \pm 1[5]~{\rm ou}~p \equiv \pm 1 [8]\}}{{\rm card}(\{p \in \mathbb{P},~p \leq x\}}~?\]
\noindent Pour cela, posons \[D(a,n,x):=\frac{{\rm card}(\{p \in \mathbb{P} \cap [0,x],~p \equiv a[n]\}}{{\rm card}(\{p \in \mathbb{P},~p \leq x\}}.\] Par le principe d'inclusion-exclusion, on a :
\begin{eqnarray*}
D(x) &=& D(1,5,x)+D(-1,5,x)+D(1,8,x)+D(-1,8,x)-D(1,40,x) \\
     & & -D(-1,40,x)-D(9,40,x)-D(-9,40,x).\\
\end{eqnarray*}

\noindent Or, par le théorème de De La Vallée Poussin (version quantitative du théorème de la progression arithmétique de Dirichlet, voir \cite{DVP} et \cite{Ro} Théorème 4.8), on a, si $a$ et $n$ sont premiers entre eux, \[\lim\limits_{x\rightarrow +\infty} D(a,n,x) = \frac{1}{\varphi(n)},\] avec $\varphi$ la fonction indicatrice d'Euler. 
\\
\\Donc, \[\lim\limits_{x\rightarrow +\infty} D(x) = \frac{1}{4}+\frac{1}{4}+\frac{1}{4}+\frac{1}{4}-\frac{1}{16}-\frac{1}{16}-\frac{1}{16}-\frac{1}{16}=\frac{3}{4}.\]
\\
\\Ainsi, les propositions précédentes permettent de considérer beaucoup d'entiers composés (voir aussi annexe \ref{A}). Cependant, il reste encore plusieurs "petits" nombres premiers pour lesquels nous n'avons pas encore de résultat. Pour remédier à cela, nous allons maintenant considérer certains de ces nombres séparément, en nous concentrant, en priorité, sur ceux inférieurs à 100.

\begin{proposition}
\label{310}

Soit $N \geq 2$ non premier et $p$ un nombre premier divisant $N$. Si $p$ appartient à l'ensemble \[\{13, 43, 83, 197, 293, 307, 547, 587, 643, 757, 797, 827, 853, 883\}\] alors $N$ est monomialement réductible.

\end{proposition}

\begin{proof}

Comme précédemment, on peut écrire $N=mp$ avec $m$ et $p$ premiers entre eux et $m \geq 2$. Montrons que la condition de l'énoncé suffit à avoir l'existence d'un $\overline{k}$ pour lequel la solution $\overline{k}$-monomiale minimale peut être réduite à l'aide d'une solution de taille 9.
 \[K_{7}(\overline{k},\ldots,\overline{k})+\overline{1}=\overline{k}^{7}-6\overline{k^{5}}+10\overline{k}^{3}-4\overline{k}+\overline{1}=(\overline{k}+\overline{1})(\overline{k}^{3}-3\overline{k}+\overline{1})(\overline{k}^{3}-\overline{k}^{2}-2\overline{k}+\overline{1}).\]  
Le tableau suivant donne une racine modulo $p$ de $X^{3}-X^{2}-2X+1$, notée $\overline{k_{p}}$, pour chaque $p$ de la liste.

\begin{center}
\begin{tabular}{|c|c|c|c|c|c|c|c|c|}
  \hline
  $p$   & 13 & 43 & 83 & 197 & 293 & 307 & 547 & 587   \rule[-7pt]{0pt}{20pt} \\
  \hline
  $\overline{k_{p}}$  & $\overline{3}$ & $\overline{24}$ & $\overline{68}$ & $\overline{39}$ & $\overline{16}$ & $\overline{24}$ & $\overline{234}$ & $\overline{63}$  \rule[-7pt]{0pt}{20pt} \\
  \hline
\end{tabular}
\end{center}

\begin{center}
\begin{tabular}{|c|c|c|c|c|c|c|}
  \hline
  $p$  & 643 & 757 & 797 & 827 & 853 & 883 \rule[-7pt]{0pt}{20pt} \\
  \hline
  $\overline{k_{p}}$  & $\overline{245}$  & $\overline{431}$ & $\overline{363}$ & $\overline{168}$ & $\overline{153}$ & $\overline{391}$ \rule[-7pt]{0pt}{20pt} \\
  \hline
\end{tabular}
\end{center}

\noindent Soit $x$ une racine de $X^{3}-X^{2}-2X+1$. On a \[K_{4}(x,x,x,x)=x^{4}-3x^{2}+1=x^{3}+2x^{2}-x-3x^{2}+1=x^{3}-x^{2}-x+1=x\] et

\begin{eqnarray*}
K_{8}(x,\ldots,x) &=& xK_{7}(x,\ldots,x)-K_{6}(x,\ldots,x) \\
                  &=& -x-xK_{5}(x,\ldots,x)+K_{4}(x,x,x,x) \\
									&=& -x-x^{2}K_{4}(x,x,x,x)+xK_{3}(x,x,x)+x \\
									&=& -x^{3}+x^{4}-2x^{2} \\
									&=& x(x^{3}-x^{2}-2x) \\
									&=& -x. \\
\end{eqnarray*}

\noindent On pose $\overline{l}=\psi(-1+m\mathbb{Z}, k_{p}+p\mathbb{Z})$ ($m$ et $p$ premiers entre eux). Comme $\psi$ est un morphisme d'anneaux unitaires, on a \[K_{7}(\overline{l},\ldots,\overline{l})=-\overline{1}.\] On pose \[\overline{a}=-K_{6}(\overline{l},\ldots,\overline{l}).\] Par le lemme \ref{34}, $(\overline{a},\overline{l},\ldots,\overline{l},\overline{a})$ est une solution de taille 9 de \eqref{p}.
\\
\\Il nous reste maintenant à prouver que la solution $\overline{l}$-monomiale minimale de \eqref{p} est de taille supérieure à 10. Notons $h$ cette taille. Comme précédemment, $h$ est un multiple de 3. 
\begin{itemize}
\item $h \neq 3$. En effet, si $h=3$ alors la solution $(k_{p}+p\mathbb{Z})$-monomiale minimale de $(E_{p})$ est de taille 3. Or, ceci est impossible car $k_{p}+p\mathbb{Z}\neq \pm 1 +p\mathbb{Z}$.
\item $h \neq 6$. En effet, si $h=6$ alors $k_{p}+p\mathbb{Z}=K_{4}(k_{p}+p\mathbb{Z},\ldots,k_{p}+p\mathbb{Z})=\pm 1+p\mathbb{Z}$ ce qui n'est pas le cas d'après le tableau ci-dessus. 
\item $h \neq 9$. En effet, si $h=9$ alors $-k_{p}+p\mathbb{Z}=K_{8}(k_{p}+p\mathbb{Z},\ldots,k_{p}+p\mathbb{Z})=0+p\mathbb{Z}$ ce qui n'est pas le cas d'après le tableau ci-dessus. 
\\
\end{itemize}
\noindent Ainsi, $h \geq 10$. Donc, la solution $\overline{l}$-monomiale minimale de \eqref{p} est réductible et $N$ est monomialement réductible.

\end{proof}

\begin{remark}

{\rm Si $x$ est une racine de $X^{3}-3X+1$ alors $K_{8}(x,\ldots,x)=0$.}

\end{remark}

\noindent De la même façon, on a aussi le résultat suivant :

\begin{proposition}
\label{311}

Soit $N \geq 2$ non premier et $p$ un nombre premier divisant $N$. Si $p$ appartient à l'ensemble \[\{67, 373, 397, 683, 947\}\] alors $N$ est monomialement réductible.

\end{proposition}

\begin{proof}

Comme précédemment, on peut écrire $N=mp$ avec $m$ et $p$ premiers entre eux et $m \geq 2$. Montrons que la condition de l'énoncé suffit à avoir l'existence d'un $\overline{k}$ pour lequel la solution $\overline{k}$-monomiale minimale peut être réduite à l'aide d'une solution de taille 11.
 \begin{eqnarray*}
K_{9}(\overline{k},\ldots,\overline{k})+\overline{1} &=& \overline{k}^{9}-8\overline{k^{7}}+21\overline{k}^{5}-20\overline{k}^{3}+5\overline{k}+\overline{1} \\
                                                     &=& (\overline{k}-\overline{1})(\overline{k}^{3}-3\overline{k}-\overline{1})(\overline{k}^{5}+\overline{k}^{4}-4\overline{k}^{3}-3\overline{k}^{2}+3\overline{k}+\overline{1}).\\
\end{eqnarray*}
																											
\noindent Le tableau suivant donne une racine modulo $p$ de $X^{5}+X^{4}-4X^{3}-3X^{2}+3X+1$, notée $\overline{k_{p}}$, pour chaque $p$ de la liste.

\begin{center}
\begin{tabular}{|c|c|c|c|c|c|}
  \hline
  $p$  & 67 & 373 & 397 & 683 & 947  \rule[-7pt]{0pt}{20pt} \\
  \hline
  $\overline{k_{p}}$  & $\overline{17}$ & $\overline{46}$ & $\overline{95}$ & $\overline{32}$ & $\overline{129}$  \rule[-7pt]{0pt}{20pt} \\
  \hline
	$K_{4}(\overline{k_{p}},\ldots,\overline{k_{p}})$ & $\overline{44}$ & $\overline{331}$ & $\overline{42}$ & $\overline{515}$ & $\overline{463}$     \rule[-7pt]{0pt}{20pt} \\
	\hline
	$K_{8}(\overline{k_{p}},\ldots,\overline{k_{p}})$ & $\overline{50}$ & $\overline{327}$ & $\overline{302}$ & $\overline{651}$ & $\overline{818}$      \rule[-7pt]{0pt}{20pt} \\
  \hline
\end{tabular}
\end{center}

\noindent On pose $\overline{l}=\psi(1+m\mathbb{Z}, k_{p}+p\mathbb{Z})$ ($m$ et $p$ premiers entre eux). Comme $\psi$ est un morphisme d'anneaux unitaires, on a \[K_{9}(\overline{l},\ldots,\overline{l})=-\overline{1}.\] On pose \[\overline{a}=-K_{8}(\overline{l},\ldots,\overline{l}).\] Par le lemme \ref{34}, $(\overline{a},\overline{l},\ldots,\overline{l},\overline{a})$ est une solution de taille 11 de \eqref{p}.
\\
\\Il nous reste maintenant à prouver que la solution $\overline{l}$-monomiale minimale de \eqref{p} est de taille supérieure à 12. Notons $h$ cette taille. Comme précédemment, $h$ est un multiple de 3. 
\begin{itemize}
\item $h \neq 3$. En effet, si $h=3$ alors la solution $(k_{p}+p\mathbb{Z})$-monomiale minimale de $(E_{p})$ est de taille 3. Or, ceci est impossible car $k_{p}+p\mathbb{Z}\neq \pm 1 +p\mathbb{Z}$.
\item $h \neq 6$. En effet, si $h=6$ alors $K_{4}(k_{p}+p\mathbb{Z},\ldots,k_{p}+p\mathbb{Z})=\pm 1+p\mathbb{Z}$ ce qui n'est pas le cas d'après le tableau ci-dessus. 
\item $h \neq 9$. En effet, si $h=9$ alors $K_{8}(k_{p}+p\mathbb{Z},\ldots,k_{p}+p\mathbb{Z})=0+p\mathbb{Z}$ ce qui n'est pas le cas d'après le tableau ci-dessus. 
\\
\end{itemize}
\noindent Ainsi, $h \geq 12$. Donc, la solution $\overline{l}$-monomiale minimale de \eqref{p} est réductible et $N$ est monomialement réductible.

\end{proof}

\noindent On termine par les deux résultats ci-dessous :

\begin{proposition}
\label{312}

Soit $N \geq 2$ non premier et $p$ un nombre premier divisant $N$. Si $p$ appartient à l'ensemble \[\{53, 157, 443, 467, 677\}\] alors $N$ est monomialement réductible.

\end{proposition}

\begin{proof}

Comme précédemment, on peut écrire $N=mp$ avec $m$ et $p$ premiers entre eux et $m \geq 2$. Montrons que la condition de l'énoncé suffit à avoir l'existence d'un $\overline{k}$ pour lequel la solution $\overline{k}$-monomiale minimale peut être réduite à l'aide d'une solution de taille 15.
 \begin{eqnarray*}
K_{13}(\overline{k},\ldots,\overline{k})+\overline{1} &=& \overline{k}^{13}-12\overline{k^{11}}+55\overline{k}^{9}-120\overline{k}^{7}+126\overline{k}^{5}-56\overline{k}^{3}+7\overline{k}+\overline{1} \\
                                                      &=& (\overline{k}+\overline{1})(\overline{k}^{2}+\overline{k}-\overline{1})(\overline{k}^{4}-\overline{k}^{3}-4\overline{k}^{2}+4\overline{k}+\overline{1})\\
																										  & &(\overline{k}^{6}-\overline{k}^{5}-5\overline{k}^{4}+4\overline{k}^{3}+6\overline{k}^{2}-3\overline{k}-\overline{1}).\\
\end{eqnarray*}
																											
\noindent Le tableau suivant donne une racine modulo $p$ de $X^{6}-X^{5}-5X^{4}+4X^{3}+6X^{2}-3X-1$, notée $\overline{k_{p}}$, pour chaque $p$ de la liste.

\begin{center}
\begin{tabular}{|c|c|c|c|c|c|}
  \hline
  $p$  & 53 & 157 & 443 & 467 & 677  \rule[-7pt]{0pt}{20pt} \\
  \hline
  $\overline{k_{p}}$  & $\overline{15}$ & $\overline{12}$ & $\overline{17}$ & $\overline{28}$ & $\overline{98}$  \rule[-7pt]{0pt}{20pt} \\
  \hline
	$K_{4}(\overline{k_{p}},\ldots,\overline{k_{p}})$ & $\overline{25}$ & $\overline{52}$ & $\overline{257}$ & $\overline{68}$ & $\overline{605}$     \rule[-7pt]{0pt}{20pt} \\
	\hline
	$K_{8}(\overline{k_{p}},\ldots,\overline{k_{p}})$ & $\overline{6}$ & $\overline{134}$ & $\overline{6}$ & $\overline{414}$ & $\overline{643}$      \rule[-7pt]{0pt}{20pt} \\
  \hline
	$K_{10}(\overline{k_{p}},\ldots,\overline{k_{p}})$ & $\overline{15}$ & $\overline{12}$ & $\overline{17}$ & $\overline{28}$ & $\overline{98}$      \rule[-7pt]{0pt}{20pt} \\
  \hline
	$K_{14}(\overline{k_{p}},\ldots,\overline{k_{p}})$ & $\overline{38}$ & $\overline{145}$ & $\overline{426}$ & $\overline{439}$ & $\overline{579}$      \rule[-7pt]{0pt}{20pt} \\
  \hline
\end{tabular}
\end{center}

\noindent On pose $\overline{l}=\psi(-1+m\mathbb{Z}, k_{p}+p\mathbb{Z})$ ($m$ et $p$ premiers entre eux). Comme $\psi$ est un morphisme d'anneaux unitaires, on a \[K_{13}(\overline{l},\ldots,\overline{l})=-\overline{1}.\] On pose \[\overline{a}=-K_{12}(\overline{l},\ldots,\overline{l}).\] Par le lemme \ref{34}, $(\overline{a},\overline{l},\ldots,\overline{l},\overline{a})$ est une solution de taille 15 de \eqref{p}.
\\
\\Il nous reste maintenant à prouver que la solution $\overline{l}$-monomiale minimale de \eqref{p} est de taille supérieure à 16. Notons $h$ cette taille. Comme précédemment, $h$ est un multiple de 3. D'après le tableau ci-dessus $h \notin \{3, 6, 9, 12, 15\}$. 
\\
\\Ainsi, $h \geq 16$. Donc, la solution $\overline{l}$-monomiale minimale de \eqref{p} est réductible et $N$ est monomialement réductible.

\end{proof}

\begin{proposition}
\label{313}

Soit $N \geq 2$ non premier et $p$ un nombre premier divisant $N$. Si $p$ appartient à l'ensemble $\{37, 227\}$ alors $N$ est monomialement réductible.

\end{proposition}

\begin{proof}

Comme précédemment, on peut écrire $N=mp$ avec $m$ et $p$ premiers entre eux et $m \geq 2$. Montrons que la condition de l'énoncé suffit à avoir l'existence d'un $\overline{k}$ pour lequel la solution $\overline{k}$-monomiale minimale peut être réduite à l'aide d'une solution de taille 21.

\begin{eqnarray*}
K_{19}(\overline{k},\ldots,\overline{k})+\overline{1} &=& \overline{k}^{19}-18 \overline{k}^{17}+136 \overline{k}^{15}-560 \overline{k}^{13}+1365 \overline{k}^{11}-2002 \overline{k}^{9}+1716 \overline{k}^{7}\\
                                                      & &  -792 \overline{k}^{5}+165\overline{k}^{3}-10 \overline{k}+\overline{1} \\
                                                      &=& (\overline{k}+\overline{1})(\overline{k}^{3}+\overline{k}^{2}-2\overline{k}-\overline{1})(\overline{k}^{6}-\overline{k}^{5}-6\overline{k}^{4}+6\overline{k}^{3}+8\overline{k}^{2}-8\overline{k}+\overline{1}) \\
																											& & (\overline{k}^{9}-\overline{k}^{8}-8\overline{k}^{7}+7\overline{k}^{6}+21\overline{k}^{5}-15\overline{k}^{4}-20\overline{k}^{3}+10\overline{k}^{2}+5\overline{k}-\overline{1}).\\
\end{eqnarray*}
																											
Le tableau suivant donne une racine modulo $p$, que l'on note $\overline{k_{p}}$, pour chaque $p$ de la liste, du polynôme $X^{9}-X^{8}-8X^{7}+7X^{6}+21X^{5}-15X^{4}-20X^{3}+10X^{2}+5X-1$ ainsi que les $K_{i}(\overline{k_{p}},\ldots,\overline{k_{p}})$.

\begin{center}
\begin{tabular}{|c|c|c|c|c|c|c|c|c|}
  \hline
  \multicolumn{1}{|c|}{\backslashbox{$p$}{\vrule width 0pt height 1.25em$K_{i}$}} & $\overline{k_{p}}$ & $K_{3}$ & $K_{6}$ & $K_{9}$ & $K_{12}$ & $K_{15}$ & $K_{18}$ & $K_{21}$ \rule[-7pt]{0pt}{20pt} \\
  \hline
  37  & $\overline{3}$ & $\overline{21}$ & $\overline{7}$ & $\overline{31}$ & $\overline{33}$ & $\overline{8}$ & $\overline{0}$ & $\overline{29}$ \rule[-7pt]{0pt}{20pt} \\
  \hline
	227 & $\overline{8}$ & $\overline{42}$ & $\overline{65}$ & $\overline{125}$ & $\overline{99}$ & $\overline{63}$ & $\overline{0}$ & $\overline{164}$     \rule[-7pt]{0pt}{20pt} \\
	\hline
\end{tabular}
\end{center}

\noindent On pose $\overline{l}=\psi(-1+m\mathbb{Z}, k_{p}+p\mathbb{Z})$ ($m$ et $p$ premiers entre eux). Comme $\psi$ est un morphisme d'anneaux unitaires, on a \[K_{19}(\overline{l},\ldots,\overline{l})=-\overline{1}.\] On pose \[\overline{a}=-K_{18}(\overline{l},\ldots,\overline{l}).\] Par le lemme \ref{34}, $(\overline{a},\overline{l},\ldots,\overline{l},\overline{a})$ est une solution de taille 21 de \eqref{p}.
\\
\\Il nous reste maintenant à prouver que la solution $\overline{l}$-monomiale minimale de \eqref{p} est de taille supérieure à 22. Notons $h$ cette taille. Comme précédemment, $h$ est un multiple de 3. D'après le tableau ci-dessus $h \notin \{3, 6, 9, 12, 15, 18, 21\}$. 
\\
\\Ainsi, $h \geq 22$. Donc, la solution $\overline{l}$-monomiale minimale de \eqref{p} est réductible et $N$ est monomialement réductible.

\end{proof}

Pour démontrer le point iii) du théorème \ref{27}, il ne nous reste plus qu'à considérer les nombres impairs divisibles par 3.

\subsection{Le cas des nombres pairs et des multiples de 3}
\label{NP}

Les résultats présentés ci-dessus nous ont permis d'étudier la plupart des petits nombres premiers. Cela dit, 2 et 3 n'ont toujours pas été considérés (sauf dans la proposition \ref{37}). Bien entendu, il est impossible d'avoir des résultats similaires aux propositions \ref{310} à \ref{313} pour ces deux nombres puisqu'on a, par exemple, 24 monomialement irréductible (proposition \ref{35}). Notre objectif est donc d'essayer de caractériser les multiples de 2 et 3 monomialement irréductibles. 
\\
\\Afin d'alléger les notations, on pose, à partir de maintenant, $K_{n}(\overline{l}):=K_{n}(\overline{l},\ldots,\overline{l})$. On commence par le résultat technique ci-dessous :

\begin{lemma}
\label{314}

Soient $k$ et $m$ deux entiers naturels supérieurs à 2 et $N=km$. On suppose que $m \equiv \pm l [k]$ et que $l^{2} \equiv 1 [k]$. Soit $n=6q+r$ avec $q$ entier naturel et $r \in [\![0;5]\!]$.
\\
\\i) Supposons $m \equiv -l [k]$. On a $K_{n}(\overline{lm+2})=K_{r}(\overline{lm+2})+\overline{6q(lm+1)}.$ 
\\
\\ \noindent ii) Supposons $m \equiv l [k]$. On a $K_{n}(\overline{lm-2})=K_{r}(\overline{lm-2})+\overline{(-1)^{n+1}6q(lm-1)}.$ 
\\
\\ \noindent iii) Supposons, de plus, $k \geq 3$, $m$ impair et non divisible par 3. Si $m \equiv -l [k]$ (resp. $m \equiv l [k]$) alors La solution $\overline{lm+2}$ (resp. $\overline{lm-2}$) -monomiale minimale de \eqref{p} est de taille $6m$.

\end{lemma}

\begin{proof}

i) On a $\overline{m^{2}}=\overline{-lm}$ et $\overline{l^{2}m}=\overline{m}$. Calculons $K_{i}(\overline{lm+2})$ pour $i \in [\![0;5]\!]$.
\begin{itemize}
\item $K_{0}=\overline{1}$; 
\item $K_{1}(\overline{lm+2})=\overline{lm+2}$; 
\item $K_{2}(\overline{lm+2})=\overline{3lm+3}$;
\item $K_{3}(\overline{lm+2})=\overline{5lm+4}$; 
\item $K_{4}(\overline{lm+2})=\overline{6lm+5}$; 
\item $K_{5}(\overline{lm+2})=\overline{6lm+6}$.
\\
\end{itemize}

Montrons la formule par récurrence. Celle-ci est, évidemment, vraie pour $n \leq 5$. Supposons qu'il existe un $n$ tel que la formule est vraie pour $n$ et $n-1$. On a \[K_{n+1}(\overline{lm+2})= \overline{(lm+2)}K_{n}(\overline{lm+2})-K_{n-1}(\overline{lm+2}).\]
\noindent On va distinguer deux cas :

\begin{itemize}

\item Si $n=6q$.

\begin{eqnarray*}
K_{n+1}(\overline{lm+2}) &=& \overline{(lm+2)}(K_{0}+\overline{6q(lm+1)})-K_{5}(\overline{lm+2})-\overline{6(q-1)(lm+1)}\\
                         &=& \overline{lm+2+6l^{2}m^{2}q+6lmq+12lmq+12q-6lm-6-6lmq-6q+6lm+6} \\
												 &=& K_{1}(\overline{lm+2})+\overline{6lmq+6q} \\
												 &=& K_{1}(\overline{lm+2})+\overline{6q(lm+1)}. \\
\end{eqnarray*}

\item Si $n=6q+r$ avec $r \in [\![1;5]\!]$.

\begin{eqnarray*}
K_{n+1}(\overline{lm+2}) &=& \overline{(lm+2)}(K_{r}(\overline{lm+2})+\overline{6q(lm+1)})-K_{r-1}(\overline{lm+2})-\overline{6q(lm+1)} \\
                         &=& \overline{(lm+2)}K_{r}(\overline{lm+2})-K_{r-1}(\overline{lm+2})+\overline{6q(lm+1)(lm+2)} \\
												 & & \overline{-6q(lm+1)} \\
												 &=& K_{r+1}(\overline{lm+2})+\overline{6q(lm+1)^{2}} \\
												 &=& K_{r+1}(\overline{lm+2})+\overline{6q(lm+1)}. \\
\end{eqnarray*}

\noindent Si $r=5$ alors $K_{6}(\overline{lm+2})=\overline{6lm+7}=K_{0}+\overline{6lm+6}$. En combinant cela à la formule ci-dessus, on obtient $K_{n+1}(\overline{lm+2})=K_{0}+\overline{6(q+1)(lm+1)}$.
\\
\end{itemize}	

\noindent La formule est vraie pour $n+1$. Par récurrence, la formule est vraie pour tout $n$.																			
\\
\\ii) On a $\overline{m^{2}}=\overline{lm}$ et $\overline{l^{2}m}=\overline{m}$. Calculons $K_{i}(\overline{lm-2})$ pour $i \in [\![0;5]\!]$.
\begin{itemize}
\item $K_{0}=\overline{1}$; 
\item $K_{1}(\overline{lm-2})=\overline{lm-2}$; 
\item $K_{2}(\overline{lm-2})=\overline{-3lm+3}$;
\item $K_{3}(\overline{lm-2})=\overline{5lm-4}$; 
\item $K_{4}(\overline{lm-2})=\overline{-6lm+5}$; 
\item $K_{5}(\overline{lm-2})=\overline{6lm-6}$.
\\
\end{itemize}

Montrons la formule par récurrence. Celle-ci est, évidemment, vraie pour $n \leq 5$. Supposons qu'il existe un $n$ tel que la formule est vraie pour $n$ et $n-1$. On a \[K_{n+1}(\overline{lm-2})= (lm-2)K_{n}(\overline{lm-2})-K_{n-1}(\overline{lm-2}).\]

\noindent On va distinguer deux cas :

\begin{itemize}

\item Si $n=6q$. $n$ est pair et on a : 

\end{itemize}

\begin{eqnarray*}
K_{n+1}(\overline{lm-2}) &=& \overline{(lm-2)}(K_{0}+\overline{(-1)^{n+1}6q(lm-1)})-K_{5}(\overline{lm-2})-\overline{(-1)^{n}6(q-1)(lm-1)} \\
                         &=& \overline{lm-2-6lm+6+(-1)^{n+1}(-12lmq+12q+6qlm-6q-6lm+6)} \\
												 &=& K_{1}(\overline{lm-2})+\overline{(-1)^{n+1}(-6lmq+6q-6lm+6)-6lm+6} \\
												 &=& K_{1}(\overline{lm-2})+\overline{(-1)^{n+2}6q(lm-1)+(-6lm+6)((-1)^{n+1}+1)} \\
												 &=& K_{1}(\overline{lm-2})+\overline{(-1)^{n+2}6q(lm-1)}~({\rm car}~n~{\rm est~pair}). \\
\end{eqnarray*}

\begin{itemize}

\item Si $n=6q+r$ avec $r \in [\![1;5]\!]$.

\begin{eqnarray*}
K_{n+1}(\overline{lm-2}) &=& \overline{(lm-2)}(K_{r}(\overline{lm-2})+\overline{(-1)^{n+1}6q(lm-1)})-K_{r-1}(\overline{lm-2}) \\
                         & & -\overline{(-1)^{n}6q(lm-1)} \\
                         &=& K_{r+1}(\overline{lm-2})+\overline{(-1)^{n+1}6q(lm-1)(lm-2+1)} \\
												 &=& K_{r+1}(\overline{lm-2})+\overline{(-1)^{n+1}6q(-lm+1)} \\
												 &=& K_{r+1}(\overline{lm-2})+\overline{(-1)^{n+2}6q(lm-1)}. \\
\end{eqnarray*}

\noindent Si $r=5$ alors $n$ est impair. Comme $n+2$ est impair, on a \[K_{6}(\overline{lm-2})=\overline{-6lm+7}=K_{0}+\overline{6-6lm}=K_{0}+(-1)^{n+2}\overline{6(lm-1)}.\] 
\noindent En combinant cela à la formule ci-dessus, on obtient \[K_{n+1}(\overline{lm-2})=K_{0}+\overline{(-1)^{n+2}6(q+1)(lm-1)}.\]
\end{itemize}	

\noindent La formule est vraie pour $n+1$. Par récurrence, celle-ci est vraie pour tout $n$.
\\
\\ \noindent iii) Supposons, de plus, $k \geq 3$, $m$ impair et non divisible par 3. On suppose que $m \equiv -l [k]$. Remarquons que $k$ et $m$ sont premiers entre eux. En effet, si $u$ divise $k$ et $m$ alors $u$ divise $l$ puisque $m \equiv -l [k]$. Donc, $u$ divise $l^{2}$. Comme $l^{2} \equiv 1 [k]$, $u$ divise 1.
\begin{itemize}
\item $lm+2 \equiv 2 [m]$ et la solution $(2+m\mathbb{Z})$-monomiale minimale de $(E_{m})$ est de taille $m$ avec 
\\$M_{m}(2+m\mathbb{Z},\ldots,2+m\mathbb{Z})=Id$;
\item $lm+2 \equiv 1 [k]$ et la solution $(1+k\mathbb{Z})$-monomiale minimale de $(E_{k})$ est de taille $3$ avec 
\\$M_{3}(1+k\mathbb{Z},1+k\mathbb{Z},1+k\mathbb{Z})=-Id$.
\\
\end{itemize}

\noindent Comme $k$ et $m$ divisent $N$ on a que la taille de la solution $\overline{(lm+2)}$-monomiale minimale de \eqref{p} est un multiple de ${\rm ppcm}(3,m)=3m$ (car $m$ n'est pas divisible par 3). De plus, \[M_{3m}(2+m\mathbb{Z},\ldots,2+m\mathbb{Z})=Id\] et \[M_{3m}(1+k\mathbb{Z},\ldots,1+k\mathbb{Z})=-Id \neq Id ~{\rm (puisque}~m ~{\rm est~impair~et}~k \geq 3).\] 
\noindent Donc, la solution n'est pas de taille $3m$. En revanche, on a	\[M_{6m}(2+m\mathbb{Z},\ldots,2+m\mathbb{Z})=Id\] et \[M_{6m}(1+k\mathbb{Z},\ldots,1+k\mathbb{Z})=Id.\] 

\noindent Ainsi, par le lemme chinois ($k$ et $m$ premiers entre eux), la solution est de taille $6m$. On procède de même si $m \equiv l [k]$.

\end{proof}

\noindent On donne en annexe \ref{B} les formules, lorsque les conditions initiales du lemme sont vérifiées, de $K_{n}(\overline{lm+2})$ et $K_{n}(\overline{lm-2})$ pour les petites valeurs de $n$. On peut maintenant démontrer le résultat qui suit.

\begin{proposition}
\label{315}

Soient $m$ un entier naturel supérieur à 2, impair, non divisible par 3. Il n'y pas d'entier monomialement irréductible de la forme $2m$.

\end{proposition}

\begin{proof}

Soit $N=2m$ avec $m$ vérifiant les conditions de l'énoncé. Comme $m$ est impair et non divisible par 3, on a $m \equiv \pm 1[6]$ et $m \equiv -1 [2]$. Par le lemme précédent, $K_{n}(\overline{m+2})=K_{r}(\overline{m+2})+\overline{n-r}$ (puisque $\overline{6m}=\overline{0}$ et $6q=n-r$). En procédant comme dans la démonstration précédente, on obtient que la solution $\overline{(m+2)}$-monomiale minimale de \eqref{p} est de taille $3m$.

\begin{itemize}

\item Si $m \equiv 1 [6]$ alors 
\[K_{m}(\overline{m+2})=K_{1}(\overline{m+2})+\overline{(m-1)}=\overline{m+2+m-1}=\overline{1}.\]
On pose $\overline{a}=K_{m-1}(\overline{m+2}).$ Par le lemme \ref{34}, $(\overline{a},\overline{m+2},\ldots,\overline{m+2},\overline{a})$ est une solution de taille $m+2$ de \eqref{p}.
\\
\item Si $m \equiv 5 [6]$ alors 
\[K_{m-2}(\overline{m+2})=K_{3}(\overline{m+2})+\overline{(m-2-3)}=\overline{5m+4+m-5}=\overline{-1}.\]
On pose $\overline{a}=-K_{m-3}(\overline{m+2}).$ Par le lemme \ref{34}, $(\overline{a},\overline{m+2},\ldots,\overline{m+2},\overline{a})$ est une solution de taille $m$ de \eqref{p}.
\\
\end{itemize}

\noindent Ainsi, la solution  $\overline{(m+2)}$-monomiale minimale de \eqref{p} est réductible. Donc, $N$ est monomialement réductible.

\end{proof}

\begin{remark}

{\rm Si $N=2m$ avec $m$ impair mais divisible par 3 alors la solution $\overline{(m+2)}$-monomiale minimale de \eqref{p} est irréductible de taille $m$. Démontrons-le rapidement. 
\\
\\Si $N=6=2 \times 3$ alors $\overline{(m+2)}=\overline{-1}$ et la solution est irréductible de taille 3. 
\\
\\Si $N \geq 18$ alors $\overline{(m+2)} \neq \pm \overline{1}$. En procédant comme dans la démonstration du lemme \ref{314}, on a que la solution $\overline{(m+2)}$-monomiale minimale de \eqref{p} est de taille $m$ (car $m$ est divisible par 3). Si la solution $\overline{(m+2)}$-monomiale minimale de \eqref{p} est réductible alors il existe une solution de \eqref{p} de la forme $(\overline{a},\overline{m+2},\ldots,\overline{m+2},\overline{a})$ de taille comprise entre 4 et $m-2$. En particulier, $\exists j \in [\![2;m-4]\!]$ tel que $K_{j}(\overline{m+2},\ldots,\overline{m+2})=\pm \overline{1}$. Considérons les différents cas :
\begin{itemize}
\item si $j \equiv u [6]$ avec $u \in \{0, 4, 5\}$, $K_{j}(\overline{m+2})=\overline{u+1+j-u}=\overline{1+j} \neq \pm \overline{1}$ car $1+j \in [\![3;m-3]\!]$;
\item si $j \equiv v [6]$ avec $v \in [\![1;3]\!]$, $K_{j}(\overline{m+2},\ldots,\overline{m+2})=\overline{m+v+1+j-v}=\overline{m+j+1} \neq \pm \overline{1}$ car $m+j+1 \in [\![m+3;2m-3]\!]$.
\\
\end{itemize}
\noindent Ainsi, la solution $\overline{(m+2)}$-monomiale minimale de \eqref{p} est irréductible.

}

\end{remark}

\noindent Pour continuer notre étude, on a besoin d'une version légèrement différente du lemme \ref{314}.

\begin{lemma}
\label{316}

Soient $k$ et $m$ deux entiers naturels supérieurs à 2 et $N=km$. On suppose que $m \equiv \pm l [k]$ et que $l^{2} \equiv 1 [k]$. 
\\
\\A) Soit $n=12q+r$ avec $q$ entier naturel et $r \in [\![0;11]\!]$.
\\
\\i) Supposons $m \equiv -l [k]$. On a $K_{n}(\overline{lm+2})=K_{r}(\overline{lm+2})+\overline{12q(lm+1)}.$ 
\\ \noindent ii) Supposons $m \equiv l [k]$. On a $K_{n}(\overline{lm-2})=K_{r}(\overline{lm-2})+\overline{(-1)^{n+1}12q(lm-1)}.$ 
\\
\\B) Soit $n=24q+r$ avec $q$ entier naturel et $r \in [\![0;23]\!]$.
\\
\\i) Supposons $m \equiv -l [k]$. On a $K_{n}(\overline{lm+2})=K_{r}(\overline{lm+2})+\overline{24q(lm+1)}.$ 
\\ \noindent ii) Supposons $m \equiv l [k]$. On a $K_{n}(\overline{lm-2})=K_{r}(\overline{lm-2})+\overline{(-1)^{n+1}24q(lm-1)}.$ 

\end{lemma}

\begin{proof}

A) On se place dans le cas i). Si $r \in [\![0;5]\!]$ alors, par le lemme \ref{314}, \[K_{n}(\overline{lm+2})=K_{r}(\overline{lm+2})+\overline{6(2q)(lm+1)}=K_{r}(\overline{lm+2})+\overline{12q(lm+1)}.\] Si $r \in [\![6;11]\!]$, $n=6(2q+1)+(r-6)$ et, par le lemme \ref{314}, on obtient, \[K_{n}(\overline{lm+2})=K_{r-6}(\overline{lm+2})+\overline{6(2q+1)(lm+1)}\] et \[K_{r}(\overline{lm+2})=K_{r-6}(\overline{lm+2})+\overline{6(lm+1)}.\] \noindent Donc, $K_{n}(\overline{lm+2})=K_{r}(\overline{lm+2})+\overline{12q(lm+1)}.$ On procède de façon analogue pour ii), en utilisant, en plus, que $n$ et $r$ ont même parité.
\\
\\B) On démontre les formules B) i) et ii) à partir des formules A) en procédant de la même façon qu'au dessus.

\end{proof}

\noindent On peut maintenant considérer les cas manquants.

\begin{proposition}
\label{317}

Soit $k \in \{3, 4, 6, 8, 12, 24\}$. Soit $m$ un entier naturel impair supérieur à 2 et non divisible par 3. Il n'y pas d'entier monomialement irréductible de la forme $km$.

\end{proposition}

\begin{proof}

Soit $N=km$ avec $k$ et $m$ vérifiant les conditions de l'énoncé. Comme $m$ est impair et non divisible par 3, on a $m \equiv \pm l [24]$ avec $l \in \{1, 5, 7, 11\}$. En particulier, si $l$ appartient à l'ensemble précédent alors $l^{2} \equiv 1 [24]$. Donc, $m \equiv \pm l [k]$ avec $l \in \{1, 5, 7, 11\}$ et $l^{2} \equiv 1 [k]$. Soit $n=24q+r$ avec $0 \leq r \leq 23$. Par les lemmes \ref{314} et \ref{316}, on a 
\begin{itemize}
\item si $m \equiv -l [24]$ alors $m \equiv -l [k]$ et $K_{n}(\overline{lm+2})=K_{r}(\overline{lm+2})+\overline{(n-r)}$. De plus, la solution $\overline{(lm+2)}$-monomiale minimale de \eqref{p} est de taille $6m$.
\item si $m \equiv l [24]$ alors $m \equiv l [k]$ et $K_{n}(\overline{lm-2})=K_{r}(\overline{lm-2})+\overline{(-1)^{n}(n-r)}$. De plus, la solution $\overline{(lm-2)}$-monomiale minimale de \eqref{p} est de taille $6m$.
\\
\end{itemize}
On traite séparément les différents cas.

\begin{itemize}

\item Si $m \equiv \pm 1 [24]$.
\begin{itemize}[label=$\circ$]
\item Si $m \equiv -1 [24]$. 

\[K_{m-2}(\overline{m+2})=K_{21}(\overline{m+2})+\overline{(m-2-21)}=\overline{23m+22+m-23}=\overline{-1}.\]
On pose $\overline{a}=-K_{m-3}(\overline{m+2}).$ Par le lemme \ref{34}, $(\overline{a},\overline{m+2},\ldots,\overline{m+2},\overline{a})$ est une solution de taille $m$ de \eqref{p}.
\\

\item Si $m \equiv 1 [24]$. 

\[K_{m}(\overline{m-2})=K_{1}(\overline{m-2})-\overline{(m-1)}=\overline{m-2-m+1}=\overline{-1}.\]
On pose $\overline{a}=-K_{m-1}(\overline{m-2}).$ Par le lemme \ref{34}, $(\overline{a},\overline{m-2},\ldots,\overline{m-2},\overline{a})$ est une solution de taille $m+2$ de \eqref{p}.
\\
\end{itemize}

\item Si $m \equiv \pm 5 [24]$. 
\begin{itemize}[label=$\circ$]
\item Si $m \equiv -5 [24]$. 

\[K_{m}(\overline{5m+2})=K_{19}(\overline{5m+2})+\overline{(m-19)}=\overline{-m+20+m-19}=\overline{1}.\]
On pose $\overline{a}=K_{m-1}(\overline{5m+2}).$ Par le lemme \ref{34}, $(\overline{a},\overline{5m+2},\ldots,\overline{5m+2},\overline{a})$ est une solution de taille $m+2$ de \eqref{p}.
\\

\item Si $m \equiv 5 [24]$. 

\[K_{m-2}(\overline{5m-2})=K_{3}(\overline{5m-2})-\overline{(m-2-3)}=\overline{m-4-m+5}=\overline{1}.\]
On pose $\overline{a}=K_{m-3}(\overline{5m-2}).$ Par le lemme \ref{34}, $(\overline{a},\overline{5m-2},\ldots,\overline{5m-2},\overline{a})$ est une solution de taille $m$ de \eqref{p}.
\\
\end{itemize}
\item Si $m \equiv \pm 7 [24]$. 
\begin{itemize}[label=$\circ$]
\item Si $m \equiv -7 [24]$. 

\[K_{m-2}(\overline{7m+2})=K_{15}(\overline{7m+2})+\overline{(m-2-15)}=\overline{-m+16+m-17}=\overline{-1}.\]
On pose $\overline{a}=-K_{m-3}(\overline{7m+2}).$ Par le lemme \ref{34}, $(\overline{a},\overline{7m+2},\ldots,\overline{7m+2},\overline{a})$ est une solution de taille $m$ de \eqref{p}.
\\

\item Si $m \equiv 7 [24]$. 

\[K_{m}(\overline{7m-2})=K_{7}(\overline{7m-2})-\overline{(m-7)}=\overline{m-8-m+7}=\overline{-1}.\]
On pose $\overline{a}=-K_{m-1}(\overline{7m-2}).$ Par le lemme \ref{34}, $(\overline{a},\overline{7m-2},\ldots,\overline{7m-2},\overline{a})$ est une solution de taille $m+2$ de \eqref{p}.
\\
\end{itemize}
\item Si $m \equiv \pm 11 [24]$. 
\begin{itemize}[label=$\circ$]
\item Si $m \equiv -11 [24]$. 

\[K_{m}(\overline{11m+2})=K_{13}(\overline{11m+2})+\overline{(m-13)}=\overline{-m+14+m-13}=\overline{1}.\]
On pose $\overline{a}=K_{m-1}(\overline{11m+2}).$ Par le lemme \ref{34}, $(\overline{a},\overline{11m+2},\ldots,\overline{11m+2},\overline{a})$ est une solution de taille $m+2$ de \eqref{p}.
\\

\item Si $m \equiv 11 [24]$. 

\[K_{m-2}(\overline{11m-2})=K_{9}(\overline{11m-2})-\overline{(m-2-9)}=\overline{m-10-m+11}=\overline{1}.\]
On pose $\overline{a}=K_{m-3}(\overline{11m-2}).$ Par le lemme \ref{34}, $(\overline{a},\overline{11m-2},\ldots,\overline{11m-2},\overline{a})$ est une solution de taille $m$ de \eqref{p}.
\\

\end{itemize}

\end{itemize}

\noindent Donc, $N$ est monomialement réductible.

\end{proof}

\begin{remark}

{\rm On ne peut pas se servir des solutions utilisées dans la démonstration précédente pour considérer tous les entiers. Par exemple, si $N=11m$, il y a des valeurs de $m$ pour lesquels la méthode ci-dessus ne convient pas. Considérons $m=5 \equiv 5 [11]$. Ici, on a $lm-2=23$. La solution $\overline{23}$-monomiale minimale de $(E_{55})$ est irréductible. En effet, celle-ci est de taille 15, les racines modulo 55 de $X(X-23)$ sont $\overline{0}$, $\overline{23}$, $\overline{33}$ et $\overline{45}$ et on montre (en effectuant les calculs) qu'il n'y a pas de solutions de $(E_{55})$ de la forme $(\overline{33},\overline{23},\ldots,\overline{23},\overline{33})$ de taille inférieure à 13.
}

\end{remark}

\subsection{Quelques éléments pour conclure}
\label{CC}

Les éléments donnés dans la section précédente permettent d'achever la preuve du théorème \ref{27}. En particulier, celui-ci permet d'avoir le résultat ci-dessous :

\begin{corollary}

Les entiers monomialement irréductibles inférieurs à 17440 sont premiers ou égaux à 4, 6, 8, 12 ou 24. 

\end{corollary}

Les résultats précédents permettent donc d'obtenir un certain nombre de conditions nécessaires pour qu'un entier impair soit monomialement irréductible. De plus, ces conditions portant sur l'ensemble des facteurs premiers d'un entier, on a résolu la question initiale pour "beaucoup" d'entiers. Les techniques et résultats développés ici peuvent bien sûr être encore raffinés. À titre d'exemple, 1093 peut être considéré avec la méthode exposée dans la proposition \ref{310} ($361$ étant une racine modulo 1093 de $X^{3}-X^{2}-2X+1$). Cependant, ces techniques de réduction ne pourront sans doute pas permettre de résoudre entièrement la question de la classification posée initialement. On peut, néanmoins, au vu des résultats obtenus, énoncer la conjecture suivante :

\begin{con}

Soit $N \geq 2$. $N$ est monomialement irréductible si et seulement si $N \in \{4, 6, 8, 12, 24\}$ ou si $N$ est premier.

\end{con}

Informatiquement, on a obtenu (grâce à un programme reprenant la méthode utilisée dans les propositions \ref{310} à \ref{313}) que les entiers $N$, non premiers, divisibles par un nombre premier compris entre 5 et 10~000 sont monomialement réductibles. Ainsi, les deux plus petits nombres premiers pour lesquels on ne dispose pas d'information sont 10~037 et 10~067. En particulier, cela permet d'avoir, informatiquement, que la conjecture ci-dessus est vraie pour $N \leq 101~042~478$.

\begin{remark}

{\rm Dans la méthode évoquée ci-dessus, on a utilisé que $\pm 1$ est racine dans $\mathbb{Z}$ d'un polynôme du type $K_{n}(X,\ldots,X)\pm 1$. Si l'on souhaite aller plus loin avec cette méthode, il est intéressant de chercher les entiers naturels $n$ pour lesquels cette propriété est vraie. 
\\
\\$K_{n}(X,\ldots,X)=U_{n}(\frac{X}{2})$ où $U_{n}$ est le $n^{i\grave{e}me}$ polynôme de Tchebychev de seconde espèce (voir \cite{CO} section 5). 
\\
\\De plus, si $\theta \in \mathbb{R}-\pi \mathbb{Z}$, $U_{n}({\rm cos}(\theta))=\frac{{\rm sin}((n+1)\theta)}{{\rm sin}(\theta)}$ (voir \cite{HM} définition 1.2). Donc, 
\[K_{n}(1,\ldots,1)=U_{n}(\frac{1}{2})=U_{n}({\rm cos}(\frac{\pi}{3}))=\frac{{\rm sin}((n+1)\frac{\pi}{3})}{{\rm sin}(\frac{\pi}{3})}=\frac{{\rm sin}((n+1)\frac{\pi}{3})}{\frac{\sqrt{3}}{2}}.\] Ainsi, \[{\rm 1~est~racine~de}~ K_{n}(X,\ldots,X)+1 \Leftrightarrow {\rm sin}((n+1)\frac{\pi}{3})=-\frac{\sqrt{3}}{2} \Leftrightarrow n \equiv 3,4 [6].\]
On obtient de façon analogue que -1 est racine de $K_{n}(X,\ldots,X)+1$ si et seulement si $n \equiv 1 [3]$. De même, 1 (resp. -1) est racine de $K_{n}(X,\ldots,X)-1$ si et seulement si $n \equiv 0,1[6]$ (resp. $n \equiv 0[3]$).
\\
\\Donc, pour continuer la méthode exposée dans la section \ref{NI}, on doit considérer $n \not\equiv 2,5 [6]$. 
}

\end{remark}

\noindent {\bf Remerciements}.
Je remercie Valentin Ovsienko et Alain Ninet pour leur aide précieuse.

\appendix

\newpage

\section{Nombres premiers inférieurs à 1000}
\label{A}

La liste suivante donne les nombres premiers inférieurs à 1000. Les nombres écrits en {\color{blue} bleus} sont congrus à $\pm 1$ modulo 5 (ou égaux à 5). Ceux écrits en {\color{red} rouges} sont congrus à $\pm 1$ modulo 8. Les nombres écrits en {\color{ufogreen} verts} sont ceux traités séparément dans les propositions \ref{310}, \ref{311}, \ref{312} et \ref{313}. 
\\
\\ 2, 3, {\color{blue} 5}, {\color{red} 7}, {\color{blue} 11}, {\color{ufogreen} 13}, {\color{red} 17}, {\color{blue} 19}, {\color{red} 23}, {\color{blue} 29}, {\color{blue} 31}, {\color{ufogreen} 37}, {\color{blue} 41}, {\color{ufogreen} 43}, {\color{red} 47}, {\color{ufogreen} 53}, {\color{blue} 59}, {\color{blue} 61}, {\color{ufogreen} 67}, {\color{blue} 71}, {\color{red} 73}, {\color{blue} 79}, {\color{ufogreen} 83}, {\color{blue} 89}, {\color{red} 97}, {\color{blue} 101}, {\color{red} 103}, 107, {\color{blue} 109}, {\color{red} 113}, {\color{red} 127}, {\color{blue} 131}, {\color{red} 137}, {\color{blue} 139}, {\color{blue} 149}, {\color{blue} 151}, {\color{ufogreen} 157}, 163, {\color{red} 167}, 173, {\color{blue} 179}, {\color{blue} 181}, {\color{blue} 191}, {\color{red} 193}, {\color{ufogreen} 197}, {\color{blue} 199}, {\color{blue} 211}, {\color{red} 223}, {\color{ufogreen} 227}, {\color{blue} 229}, {\color{red} 233}, {\color{blue} 239}, {\color{blue} 241}, {\color{blue} 251}, {\color{red} 257}, {\color{red} 263}, {\color{blue} 269}, {\color{blue} 271}, 277, {\color{blue} 281}, 283, {\color{ufogreen} 293}, {\color{ufogreen} 307}, {\color{blue} 311}, {\color{red} 313}, 317, {\color{blue} 331}, {\color{red} 337}, 347, {\color{blue} 349}, {\color{red} 353}, {\color{blue} 359}, {\color{red} 367}, {\color{ufogreen} 373}, {\color{blue} 379}, {\color{red} 383}, {\color{blue} 389}, {\color{ufogreen} 397}, {\color{blue} 401}, {\color{blue} 409}, {\color{blue} 419}, {\color{blue} 421}, {\color{blue} 431}, {\color{red} 433}, {\color{blue} 439}, {\color{ufogreen} 443}, {\color{blue} 449}, {\color{red} 457}, {\color{blue} 461}, {\color{red} 463}, {\color{ufogreen} 467}, {\color{blue} 479}, {\color{red} 487}, {\color{blue} 491}, {\color{blue} 499}, {\color{red} 503}, {\color{blue} 509}, {\color{blue} 521}, 523, {\color{blue} 541}, {\color{ufogreen} 547}, 557, 563, {\color{blue} 569}, {\color{blue} 571}, {\color{red} 577}, {\color{ufogreen} 587}, {\color{red} 593}, {\color{blue} 599}, {\color{blue} 601}, {\color{red} 607}, 613, {\color{red} 617}, {\color{blue} 619}, {\color{blue} 631}, {\color{blue} 641}, {\color{ufogreen} 643}, {\color{red} 647}, 653, {\color{blue} 659}, {\color{blue} 661}, {\color{red} 673}, {\color{ufogreen} 677}, {\color{ufogreen} 683}, {\color{blue} 691}, {\color{blue} 701}, {\color{blue} 709}, {\color{blue} 719}, {\color{red} 727}, 733, {\color{blue} 739}, {\color{red} 743}, {\color{blue} 751}, {\color{ufogreen} 757}, {\color{blue} 761}, {\color{blue} 769}, 773, 787, {\color{ufogreen} 797}, {\color{blue} 809}, {\color{blue} 811}, {\color{blue} 821}, {\color{red} 823}, {\color{ufogreen} 827}, {\color{blue} 829}, {\color{blue} 839}, {\color{ufogreen} 853}, {\color{red} 857}, {\color{blue} 859}, {\color{red} 863}, 877, {\color{blue} 881}, {\color{ufogreen} 883}, {\color{red} 887}, 907, {\color{blue} 911}, {\color{blue} 919}, {\color{blue} 929}, {\color{red} 937}, {\color{blue} 941}, {\color{ufogreen} 947}, {\color{red} 953}, {\color{red} 967}, {\color{blue} 971}, {\color{red} 977}, {\color{red} 983}, {\color{blue} 991}, 997.

\section{Calculs de $K_{n}(\overline{lm+2})$ et $K_{n}(\overline{lm-2})$ pour les petites valeurs de $n$}
\label{B}

Soient $k$ et $m$ deux entiers naturels supérieurs à 2 et $N=km$. On suppose que $m \equiv \pm l [k]$ et que $l^{2} \equiv 1 [k]$.
\\
\\On suppose que $m \equiv -l [k]$. On calcule les $K_{n}(\overline{lm+2})$.

\begin{multicols}{2}

\center{ \noindent $K_{0}(\overline{lm+2})=\overline{1}$,
\\$K_{1}(\overline{lm+2})=\overline{lm+2}$,
\\$K_{2}(\overline{lm+2})=\overline{3lm+3}$,
\\$K_{3}(\overline{lm+2})=\overline{5lm+4}$,
\\$K_{4}(\overline{lm+2})=\overline{6lm+5}$,
\\$K_{5}(\overline{lm+2})=\overline{6lm+6}$,
\\$K_{6}(\overline{lm+2})=\overline{6lm+7}$,
\\$K_{7}(\overline{lm+2})=\overline{7lm+8}$,
\\$K_{8}(\overline{lm+2})=\overline{9lm+9}$,
\\$K_{9}(\overline{lm+2})=\overline{11lm+10}$,
\\$K_{10}(\overline{lm+2})=\overline{12lm+11}$,
\\$K_{11}(\overline{lm+2})=\overline{12lm+12}$,}

\columnbreak

\center{ \noindent $K_{12}(\overline{lm+2})=\overline{12lm+13}$,
\\$K_{13}(\overline{lm+2})=\overline{13lm+14}$,
\\$K_{14}(\overline{lm+2})=\overline{15lm+15}$,
\\$K_{15}(\overline{lm+2})=\overline{17lm+16}$,
\\$K_{16}(\overline{lm+2})=\overline{18lm+17}$,
\\$K_{17}(\overline{lm+2})=\overline{18lm+18}$,
\\$K_{18}(\overline{lm+2})=\overline{18lm+19}$,
\\$K_{19}(\overline{lm+2})=\overline{19lm+20}$,
\\$K_{20}(\overline{lm+2})=\overline{21lm+21}$,
\\$K_{21}(\overline{lm+2})=\overline{23lm+22}$,
\\$K_{22}(\overline{lm+2})=\overline{24lm+23}$,
\\$K_{23}(\overline{lm+2})=\overline{24lm+24}$.}

\end{multicols}

On suppose que $m \equiv l [k]$. On calcule les $K_{n}(\overline{lm-2})$.

\begin{multicols}{2}

\center{ \noindent $K_{0}(\overline{lm-2})=\overline{1}$,
\\$K_{1}(\overline{lm-2})=\overline{lm-2}$,
\\$K_{2}(\overline{lm-2})=\overline{-3lm+3}$,
\\$K_{3}(\overline{lm-2})=\overline{5lm-4}$,
\\$K_{4}(\overline{lm-2})=\overline{-6lm+5}$,
\\$K_{5}(\overline{lm-2})=\overline{6lm-6}$,
\\$K_{6}(\overline{lm-2})=\overline{-6lm+7}$,
\\$K_{7}(\overline{lm-2})=\overline{7lm-8}$,
\\$K_{8}(\overline{lm-2})=\overline{-9lm+9}$,
\\$K_{9}(\overline{lm-2})=\overline{11lm-10}$,
\\$K_{10}(\overline{lm-2})=\overline{-12lm+11}$,
\\$K_{11}(\overline{lm-2})=\overline{12lm-12}$,}

\columnbreak

\center{ \noindent $K_{12}(\overline{lm-2})=\overline{-12lm+13}$,
\\$K_{13}(\overline{lm-2})=\overline{13lm-14}$,
\\$K_{14}(\overline{lm-2})=\overline{-15lm+15}$,
\\$K_{15}(\overline{lm-2})=\overline{17lm-16}$,
\\$K_{16}(\overline{lm-2})=\overline{-18lm+17}$,
\\$K_{17}(\overline{lm-2})=\overline{18lm-18}$,
\\$K_{18}(\overline{lm-2})=\overline{-18lm+19}$,
\\$K_{19}(\overline{lm-2})=\overline{19lm-20}$,
\\$K_{20}(\overline{lm-2})=\overline{-21lm+21}$,
\\$K_{21}(\overline{lm-2})=\overline{23lm-22}$,
\\$K_{22}(\overline{lm-2})=\overline{-24lm+23}$,
\\$K_{23}(\overline{lm-2})=\overline{24lm-24}$.}

\end{multicols}

\end{document}